\documentclass[11pt]{amsart}

\usepackage[utf8]{inputenc}
\usepackage[T1]{fontenc}
\usepackage{lmodern}
\usepackage[a4paper,margin=1in]{geometry}
\usepackage{microtype}
\usepackage{amsmath,amssymb,amsthm,mathtools}
\usepackage{mathrsfs}
\usepackage{enumitem}
\usepackage{booktabs}
\usepackage{longtable}
\usepackage{array}
\usepackage{graphicx}
\usepackage[hidelinks]{hyperref}
\usepackage[nameinlink,noabbrev]{cleveref}
\usepackage{tikz}
\usetikzlibrary{matrix,positioning,calc}
\usepackage{caption}

\allowdisplaybreaks[2]
\numberwithin{equation}{section}

\hypersetup{
  pdftitle={Residue Constraints in the Rank-Three Lifting Problem for Projective-Plane Incidence Matrices},
  pdfauthor={Jaehwan Kim}
}

\theoremstyle{plain}
\newtheorem{theorem}{Theorem}[section]
\newtheorem{lemma}[theorem]{Lemma}
\newtheorem{proposition}[theorem]{Proposition}
\newtheorem{corollary}[theorem]{Corollary}

\theoremstyle{definition}
\newtheorem{definition}[theorem]{Definition}
\newtheorem{example}[theorem]{Example}
\newtheorem{conjecture}[theorem]{Conjecture}

\theoremstyle{remark}
\newtheorem{remark}[theorem]{Remark}

\newcommand{\Kapr}{\mathrm{Kapr}}

\newcommand{\rank}{\mathrm{rank}}
\newcommand{\val}{\mathrm{val}}

\newcommand{\RR}{\mathbb{R}}

\newcommand{\Lam}{\Lambda}

\title[Residue constraints for projective-plane lifts]{Residue Constraints in the Rank-Three Lifting Problem for Projective-Plane Incidence Matrices}
\author{Jaehwan Kim}
\address{Hankuk Academy of Foreign Studies, Yongin-si, Gyeonggi-do 17035, Republic of Korea}
\email{020080@hafs.hs.kr}
\date{March 30, 2026}

\begin{document}

\begin{abstract}
We study the rank-three lifting problem for incidence matrices of finite projective
planes through residue-level determinant constraints invisible to tropical valuations
alone. In residue characteristic $\neq 3$, any rank-$\le 3$ lift of the incidence
matrix of a projective plane of order $q \ge 3$ forces $\Omega(q^8)$ distinct
admissible $2 \times 2$ zero rectangles with nontrivial residue cross-ratio.
We further prove that for $q \ge 6$ no monomial rank-$\le 3$ lift exists; in
particular, any putative low-rank lift must already involve nontrivial first-order
corrections on valuation-$0$ entries. These results arise from a local analysis
of $4 \times 4$ identity-pattern minors, where we derive the leading derangement
equation together with its first-order companion and show that every vanished
identity-pattern minor contains a cross-ratio-defective admissible rectangle.
The unresolved part of the problem is therefore genuinely global: one must decide
whether a rank-$3$ residue model, together with a compatible first-order deformation,
can satisfy the full overlapping system of local residue constraints.
\end{abstract}

\maketitle

\tableofcontents

\section{Introduction}

\subsection{Problem statement and motivation}
We study Kapranov-rank asymptotics for incidence matrices of finite projective planes, in the sense of the problem highlighted by Guterman and Shitov \cite{gutermanShitov2016}. The central difficulty is that Kapranov rank depends on cancellation among leading terms in lifts, whereas tropical rank records only the valuation-minimizing combinatorics \cite{develinSantosSturmfels2005,richterGebertSturmfelsTheobald2005,maclaganSturmfels2015,speyer2008,mikhalkin2007,einsiedlerKapranovLind2006}. The present paper does not resolve the full asymptotic problem. Instead, it isolates a residue-level obstruction package that any hypothetical rank-$\le 3$ lift must satisfy and pushes the remaining gap to one explicitly global compatibility question.


\subsection{Main contributions}
\begin{itemize}[leftmargin=2em]
  \item We derive the leading derangement equation and its first-order companion for $4\times4$ identity-pattern minors, giving an explicit local residue model for rank-$\le 3$ lifts.
  \item In residue characteristic $\neq 3$, every vanished identity-pattern minor forces a nontrivial cross-ratio on an admissible $2\times2$ zero rectangle. Counting such minors yields an $\Omega(q^8)$ lower bound on the number of distinct defective rectangles in any rank-$\le 3$ lift of $I(\Pi)$.
  \item We prove that for $q\ge 6$ no monomial rank-$\le 3$ lift exists. More sharply, one canonical $(q+1)\times(q+1)$ line-choice block already forces quadratic first-order support on valuation-$0$ entries.
  \item We reformulate the unresolved part of the rank-$3$ problem as a global compatibility question for a rank-$3$ residue realization together with its first-order deformation.
\end{itemize}

\paragraph{Scope of the result.}
The paper establishes unconditional local and counting obstructions, but it does not yet prove that $\Kapr(I(\Pi))$ grows with $q$. What remains open is a globalization step: one must rule out the possibility that a single rank-$3$ residue model and a single admissible first-order direction simultaneously satisfy the large family of overlapping local constraints produced in Sections~\ref{sec:framework}--\ref{sec:main}.

\begin{table}[htbp]
\caption{Obstruction package isolated in this paper.}
\label{tab:obstruction-package}
\centering
\small
\begin{tabular}{@{}p{0.24\textwidth}p{0.68\textwidth}@{}}
\toprule
Component & Output \\
\midrule
Identity-pattern minors & Each vanished $4\times4$ identity block contributes a nine-term derangement equation together with its first-order companion. \\
Admissible rectangles & In residue characteristic $\neq 3$, every vanished identity block forces a cross-ratio-defective admissible $2\times2$ zero rectangle. \\
Counting step & Any rank-$\le 3$ lift of $I(\Pi_q)$ with $q\ge 3$ produces $\Omega(q^8)$ distinct defective rectangles. \\
Monomial regime & For $q\ge 6$, monomial rank-$\le 3$ lifts are impossible; a canonical $(q+1)\times(q+1)$ block already forces nontrivial first-order support. \\
Remaining gap & One must globalize a rank-$3$ residue model together with a compatible first-order deformation across the overlapping local constraints. \\
\bottomrule
\end{tabular}
\end{table}

\subsection{Roadmap}
Section~\ref{sec:background} reviews projective planes and tropical/Kapranov rank, and fixes the valuation convention used throughout.
Section~\ref{sec:framework} develops the valued-field setup together with the auxiliary residue tools needed later.
Section~\ref{sec:gadgets} extracts the identity-pattern determinant identities that drive the main obstruction package.
Section~\ref{sec:main} proves the global counting theorem, the no-monomial-lift theorem, and the quadratic first-order support bound.
Section~\ref{sec:discussion} identifies the remaining global compatibility step.
Appendix~\ref{sec:app:computations} records finite computational checks that serve only as sanity tests and are not used in the proofs of the main theorems.

Although the present paper is written entirely in ordinary valued-field language, its emphasis on residue-level cancellation is methodologically aligned with the cancellation-sensitive viewpoint developed in \cite{Kim2026Hypernumber}. No result from \cite{Kim2026Hypernumber} is used in the proofs below.



\section{Background}
\label{sec:background}


\subsection{Finite projective planes and incidence matrices}
\begin{definition}[Projective plane of order $q$]
A finite projective plane of order $q$ is an incidence structure $(\mathcal{P},\mathcal{L},\in)$ such that:
(i) any two distinct points lie on a unique line, (ii) any two distinct lines meet in a unique point,
(iii) there exist four points no three collinear, and (iv) each line contains $q+1$ points (equivalently each point lies on $q+1$ lines).
Consequently $|\mathcal P|=|\mathcal L|=q^2+q+1$ \cite{dembowski1968,hirschfeld1998}.
\end{definition}

\begin{definition}[Combinatorial incidence matrix]
Fix enumerations of points $\mathcal{P}$ and lines $\mathcal{L}$ with $|\mathcal{P}|=|\mathcal{L}|=q^2+q+1$.
The incidence matrix $I(\Pi)\in\{0,1\}^{\mathcal{P}\times\mathcal{L}}$ is defined by
\[
I(\Pi)_{p,\ell}=1 \iff p\in \ell.
\]
\end{definition}

\begin{remark}[Tropical working matrix]
For determinant calculations we also use the complementary valuation pattern
\[
M(\Pi):=\mathbf 1-I(\Pi),
\]
so that incident entries have tropical weight $0$ and nonincident entries have tropical weight $1$.
\end{remark}

\begin{remark}[Convention used later]
The combinatorial object of interest is always the incidence matrix $I(\Pi)$. The auxiliary matrix $M(\Pi)=\mathbf 1-I(\Pi)$ is used only when tropical determinant bookkeeping is clearer in the complementary convention; all lift statements in Sections~\ref{sec:framework}--\ref{sec:main} are phrased back in terms of $I(\Pi)$.
\end{remark}



\subsection{Tropical linear algebra: tropical rank and Kapranov rank}
\begin{definition}[Tropical semiring]
The tropical semiring in the min-plus convention is
\[
\RR_{\mathrm{trop}}:=\RR\cup\{\infty\}, \qquad a\oplus b:=\min(a,b), \qquad a\odot b:=a+b.
\]
\end{definition}

\begin{definition}[Tropical determinant, singularity, and tropical rank]
For a square matrix $M=(m_{ij})\in \RR_{\mathrm{trop}}^{r\times r}$, its tropical determinant is
\[
\operatorname{tropdet}(M):=\min_{\sigma\in S_r}\sum_{i=1}^r m_{i,\sigma(i)}.
\]
The matrix $M$ is \emph{tropically singular} if the minimum is attained by at least two permutations; otherwise it is
\emph{tropically nonsingular}. For a rectangular matrix $A$, the tropical rank of $A$ is the largest $r$ for which $A$ has a tropically nonsingular $r\times r$ minor \cite{develinSantosSturmfels2005,joswig2021}.
\end{definition}

\begin{definition}[Kapranov rank via lifts]
Fix a valued field $(K,\val)$ over a ground field $F$. For $A\in \RR_{\mathrm{trop}}^{m\times n}$, the Kapranov rank over $F$ is
\[
\Kapr_F(A):=\min\bigl\{\rank(\widetilde A): \widetilde A\in K^{m\times n},\ \val(\widetilde A_{ij})=A_{ij}\ \text{entrywise}\bigr\}.
\]
When the ground field is understood, we write $\Kapr(A)$.
\end{definition}

\begin{remark}[Equivalent geometric formulation]
By the determinantal-ideal description together with the Fundamental Theorem of tropical geometry, the condition $\Kapr(A)\le r$ is
 equivalent to the existence of a rank-$\le r$ lift of $A$, and also to the existence of a classical linear space of dimension $r$ whose
 tropicalization contains the columns of $A$ \cite{develinSantosSturmfels2005,maclaganSturmfels2015,speyer2008}.
\end{remark}

\begin{proposition}[Initial-form criterion for tropical hypersurfaces]
Let
\[
f=\sum_u c_u x^u \in K[x_1^{\pm1},\dots,x_n^{\pm1}],
\]
and let $w\in\Gamma_{\val}^n$. If
\[
m:=\min_u\bigl(\val(c_u)+\langle w,u\rangle\bigr),
\]
then the initial form $\operatorname{in}_w(f)$ is the sum of those terms whose $w$-weight attains $m$, reduced to the residue field after
normalizing by valuation $m$. The tropical hypersurface $\operatorname{Trop}(V(f))$ is exactly the set of $w$ for which $\operatorname{in}_w(f)$ is not a
monomial \cite{maclaganSturmfels2015}.
\end{proposition}

\begin{proposition}[Basic inequality and field sensitivity]
For every tropical matrix $A$, one has
\[
\text{tropical rank}(A)\le \Kapr(A).
\]
Moreover, Kapranov rank can depend on the ground field \cite{develinSantosSturmfels2005,gutermanShitov2016,shitov2011example,shitov2012kapranov,shitov2014}.
\end{proposition}

\begin{remark}[Cancellation viewpoint]
Kapranov rank is controlled by whether the lowest-valuation determinant terms in a candidate lift can be made to cancel. Tropical data
records which permutations attain the minimum, but not whether the corresponding leading coefficients cancel in the residue field. This is
 the first point where our cancellation-layer formalism enters.
\end{remark}

\subsection{The asymptotic rank problem and known separations}
\begin{remark}[Guterman--Shitov Research Problem~5]
Fix a ground field $F$ and let $(\Pi_{q_k})$ be a sequence of projective planes of orders $q_k\to\infty$ with incidence matrices
$ I_{q_k}\in\{0,1\}^{n_k\times n_k}$, where $n_k=q_k^2+q_k+1$. The problem is to determine the asymptotic growth of
$\Kapr_F(I_{q_k})$ as $k\to\infty$ \cite{gutermanShitov2016}. In particular, one wants to know whether these ranks stay bounded,
grow polylogarithmically, or exhibit polynomial/linear growth in $q_k$ or $n_k$.
\end{remark}

\begin{remark}[Weighted versus unweighted separation]
The known constructions showing a large gap between tropical rank and Kapranov rank rely on weighted or otherwise non-projective-plane patterns; they do not by themselves settle the unweighted $0$--$1$ projective-plane incidence matrices considered here \cite{kimRoush2006,shitov2011example,shitov2012kapranov,shitov2023separation,gutermanShitov2016}.
\end{remark}

\begin{remark}[Field dependence as a structural warning]
Because Kapranov rank may depend on the ground field, any lower-bound mechanism for $I_q$ must separate valuation-level data from
 residue-level cancellation data. This is precisely the interface where a three-sign / cancellation-sensitive enrichment may add information
 beyond ordinary tropicalization \cite{maclaganSturmfels2015,joswig2021}.
\end{remark}


\subsection{Cocircuit matrices and the representability gate}
\begin{definition}[Cocircuit matrix]
Let $\mathcal M$ be a matroid on ground set $E$, and let $\mathcal C^*(\mathcal M)$ denote its set of cocircuits.
The cocircuit matrix $\mathcal C(\mathcal M)$ is the $\{0,1\}$-matrix with rows indexed by $E$, columns indexed by cocircuits,
and entries
\[
\mathcal C(\mathcal M)_{e,C^*}=
\begin{cases}
0,& e\in C^*,\\
1,& e\notin C^*.
\end{cases}
\]
\end{definition}

\begin{lemma}[Incidence matrix as cocircuit matrix]
\label{lem:incidence-as-cocircuit}
Let $\Pi=(\mathcal P,\mathcal L,\in)$ be a projective plane, and let $\mathcal M(\Pi)$ be the rank-$3$ matroid on ground set
$\mathcal P$ whose rank-$2$ flats are exactly the lines of $\Pi$. Then, up to column relabeling,
\[
I(\Pi)=\mathcal C(\mathcal M(\Pi)).
\]
\end{lemma}

\begin{proof}
In a rank-$3$ simple matroid, hyperplanes are exactly rank-$2$ flats, hence here exactly the lines $\ell\in\mathcal L$.
Cocircuits are complements of hyperplanes, so they are the sets $\mathcal P\setminus \ell$. For a point $p$ and a line $\ell$,
the cocircuit-matrix entry in the column indexed by $\mathcal P\setminus \ell$ is $1$ iff
$p\notin \mathcal P\setminus \ell$, i.e.\ iff $p\in \ell$. This is precisely the defining incidence condition for $I(\Pi)$; see also the standard oriented-matroid background in \cite{bjornerLasVergnasSturmfelsWhiteZiegler1999}.
\end{proof}

\begin{corollary}[Kapranov-rank-$3$ gate via representability]
\label{cor:kapr3-representability-gate}
Let $\Pi$ be a projective plane. Over an infinite field $k$,
\[
\Kapr_k(I(\Pi))=3
\qquad\Longleftrightarrow\qquad
\mathcal M(\Pi)\ \text{is representable over }k.
\]
In particular, if $\Pi$ is not coordinatizable over $k$, then
\[
\Kapr_k(I(\Pi))\ge 4.
\]
\end{corollary}

\begin{proof}
Combine the preceding lemma with the Develin--Santos--Sturmfels characterization of Kapranov rank for cocircuit matrices
\cite{develinSantosSturmfels2005}.
\end{proof}

\begin{remark}[Fano anchor case]
For the Fano plane, the associated rank-$3$ matroid is representable if and only if $\operatorname{char}(k)=2$.
Hence over characteristic $0$ one obtains the immediate lower bound
\[
\Kapr_k(I(F_7))\ge 4
\]
for every infinite field $k$ \cite{oxley2011Fano}.
\end{remark}

\begin{remark}[Asymptotic barrier]
The representability gate is a sharp first obstruction, but it does not yet address the asymptotic content of Research Problem~5:
it yields at best the fixed lower bound $4$ once coordinatizability fails, whereas the open question asks how $\Kapr(I(\Pi_q))$
grows with $q$.
\end{remark}


\section{Valued-field setup and auxiliary residue tools}
\label{sec:framework}


\subsection{Valued-field conventions, residues, and lift expansions}
\begin{definition}[Valued-field package and splitting]
Fix a non-archimedean valued field $(K,\val)$ with value group $\Gamma_{\val}\subseteq \RR$, valuation ring
\[
R_K:=\{x\in K: \val(x)\ge 0\},
\]
maximal ideal
\[
\mathfrak m_K:=\{x\in K: \val(x)>0\},
\]
and residue field $k:=R_K/\mathfrak m_K$. When needed, we also fix a splitting
\[
\phi: \Gamma_{\val}\to K^*, \qquad w\mapsto t^w,
\]
so that one can factor out a chosen valuation level and then reduce modulo $\mathfrak m_K$ \cite{maclaganSturmfels2015}.
\end{definition}

\begin{definition}[Leading residue]
Assume a splitting has been chosen. For $x\in K^*$ with $\val(x)=\alpha$, define its leading residue by
\[
\operatorname{lc}(x):=\overline{t^{-\alpha}x}\in k^*.
\]
If $x=0$, we set $\operatorname{lc}(0):=0$.
\end{definition}

\begin{lemma}[Residue criterion for raising valuation]
\label{lem:raise-valuation}
Let $x_1,\dots,x_m\in K$ and let
\[
\alpha:=\min_i \val(x_i), \qquad I:=\{i: \val(x_i)=\alpha\}.
\]
Assume a splitting exists. Then
\[
\val\Bigl(\sum_{i=1}^m x_i\Bigr)>\alpha
\qquad\Longleftrightarrow\qquad
\sum_{i\in I}\operatorname{lc}(x_i)=0
\quad\text{in }k.
\]
\end{lemma}

\begin{proof}
Write $x_i=t^{\val(x_i)}u_i$ with $\val(u_i)=0$ whenever $x_i\neq 0$. Then
\[
\sum_{i=1}^m x_i
=t^{\alpha}\Bigl(\sum_{i\in I} u_i + \sum_{j\notin I} t^{\val(x_j)-\alpha}u_j\Bigr).
\]
Every term in the second sum has positive valuation inside the parentheses, hence lies in $\mathfrak m_K$. Therefore the residue of
$t^{-\alpha}\sum_i x_i$ is exactly $\sum_{i\in I}\overline{u_i}=\sum_{i\in I}\operatorname{lc}(x_i)$, so the valuation rises above $\alpha$
if and only if that residue vanishes.
\end{proof}

\begin{lemma}[Leading-term cancellation criterion for minors]
\label{lem:leading-term-minor-test}
Let $A=(a_{ij})\in \RR^{k\times k}$ be a tropical matrix, and let $F=(f_{ij})\in K^{k\times k}$ be a lift with
\[
\val(f_{ij})=a_{ij}
\qquad\text{for all }i,j.
\]
Assume a splitting has been fixed, so each nonzero entry admits an expansion
\[
f_{ij}=c_{ij}t^{a_{ij}}+\text{(higher-valuation terms)}
\qquad\text{with }c_{ij}=\operatorname{lc}(f_{ij})\in k^*.
\]
For each permutation $\sigma\in S_k$, define its tropical weight and signed leading monomial coefficient by
\[
w(\sigma):=\sum_{i=1}^k a_{i,\sigma(i)},
\qquad
C(\sigma):=\operatorname{sgn}(\sigma)\prod_{i=1}^k c_{i,\sigma(i)}\in k^*.
\]
Let
\[
m:=\min_{\sigma\in S_k}w(\sigma),
\qquad
S_{\min}:=\{\sigma\in S_k:w(\sigma)=m\}.
\]
Then
\[
\det(F)=t^{m}\Bigl(\sum_{\sigma\in S_{\min}} C(\sigma)\Bigr)
+\text{(terms of valuation }>m\text{)}.
\]
In particular,
\[
\val(\det F)>m
\qquad\Longleftrightarrow\qquad
\sum_{\sigma\in S_{\min}} C(\sigma)=0
\quad\text{in }k,
\]
and if $|S_{\min}|=1$, then $\val(\det F)=m$ and $\det(F)\neq 0$.
\end{lemma}

\begin{proof}
Expand the determinant by the Leibniz formula,
\[
\det(F)=\sum_{\sigma\in S_k}\operatorname{sgn}(\sigma)\prod_{i=1}^k f_{i,\sigma(i)}.
\]
Each summand has valuation $w(\sigma)$ and leading term $C(\sigma)t^{w(\sigma)}$. Factoring out $t^m$, only the terms with
$w(\sigma)=m$ contribute to the coefficient of $t^m$, and every remaining term has strictly larger valuation. Applying
\cref{lem:raise-valuation} to the resulting sum proves the claimed criterion.
\end{proof}

\begin{remark}[Kapranov bottleneck for incidence minors]
\label{rem:kapr-bottleneck-minors}
By the lift characterization of Kapranov rank \cite{develinSantosSturmfels2005}, the condition $\Kapr(A)\le r$ means that every
$(r+1)\times(r+1)$ minor of some lift must vanish. \Cref{lem:leading-term-minor-test} isolates the precise obstruction: for each
minor one must solve a residue-field cancellation equation among all valuation-minimal permutation monomials. In projective-plane
incidence matrices, the asymptotic difficulty of Research Problem~5 is therefore not visible at the valuation level alone; it sits in the
compatibility of these leading-term cancellation constraints across many overlapping minors \cite{gutermanShitov2016}.
\end{remark}

\begin{lemma}[Two-term determinant cancellation criterion]
\label{lem:two-term-det}
Let $A=(a_{ij})\in K^{r\times r}$ and let $M=(m_{ij})$ be its valuation matrix, $m_{ij}:=\val(a_{ij})$. For each permutation
$\sigma\in S_r$, define
\[
w(\sigma):=\sum_{i=1}^r m_{i,\sigma(i)}, \qquad w_{\min}:=\min_{\pi\in S_r} w(\pi).
\]
Assume that the minimum is attained by exactly two permutations $\sigma,\tau$, and that a splitting exists so that
\[
c_{ij}:=\operatorname{lc}(a_{ij})=\overline{t^{-m_{ij}}a_{ij}}\in k^*
\]
is defined. Then
\[
\val(\det A)>w_{\min}
\qquad\Longleftrightarrow\qquad
\operatorname{sgn}(\sigma)\prod_{i=1}^r c_{i,\sigma(i)}
+ \operatorname{sgn}(\tau)\prod_{i=1}^r c_{i,\tau(i)}=0
\quad\text{in }k.
\]
\end{lemma}

\begin{proof}
Expand
\[
\det A = \sum_{\pi\in S_r}\operatorname{sgn}(\pi)\prod_{i=1}^r a_{i,\pi(i)}.
\]
Each summand has valuation $w(\pi)$. By assumption, exactly two terms have valuation $w_{\min}$ and all other terms have strictly larger
valuation. Applying \cref{lem:raise-valuation} to this sum after factoring out $t^{w_{\min}}$ gives the claim, because the leading residues
of the two minimal terms are precisely
$\operatorname{sgn}(\sigma)\prod_i c_{i,\sigma(i)}$ and $\operatorname{sgn}(\tau)\prod_i c_{i,\tau(i)}$.
\end{proof}

\begin{corollary}[Initial form in the two-minimum determinant case]
Under the hypotheses of \cref{lem:two-term-det}, the determinant initial form at valuation level $w_{\min}$ is
\[
\operatorname{in}_{w_{\min}}(\det A)=
\operatorname{sgn}(\sigma)\prod_{i=1}^r c_{i,\sigma(i)}
+
\operatorname{sgn}(\tau)\prod_{i=1}^r c_{i,\tau(i)}
\quad\text{in }k.
\]
Hence $\val(\det A)>w_{\min}$ if and only if this binomial initial form vanishes.
\end{corollary}

\begin{remark}[Connection to higher-order cancellation]
\cref{lem:two-term-det} is the first exact interface between ordinary tropical data and our cancellation-sensitive layer. Tropicalization
records that two determinant terms tie at valuation level; the residue equation records whether those tied terms actually collapse into a
single $\Lam$-type cancellation event rather than merely coexist.
\end{remark}

Given a lift $F\in K^{m\times n}$, we write its expansion (entrywise) as
\[
F = U + t V + O(t^2),
\]
where $U:=F \bmod t$ is the residue matrix and $V$ is the first correction layer.

\begin{definition}[Identity valuation pattern]
\label{def:identity-valuation-pattern}
For $k\ge 3$, let $A^{(k)}\in\{0,1\}^{k\times k}$ be the tropical matrix defined by
\[
A^{(k)}_{ii}=1,
\qquad
A^{(k)}_{ij}=0\quad (i\neq j).
\]
Thus diagonal entries have valuation $1$ and off-diagonal entries have valuation $0$.
\end{definition}

\begin{definition}[Rank-$1$ residue ansatz on the off-diagonal block]
\label{def:rank1-residue-ansatz}
Let $F=(f_{ij})$ be a lift of $A^{(k)}$. For $i\neq j$, write
\[
f_{ij}=u_{ij}+\text{(higher-valuation terms)},
\qquad u_{ij}\in k^*.
\]
We say that the off-diagonal residue data is \emph{rank-$1$} if there exist
$\alpha_1,\dots,\alpha_k\in k^*$ and $\beta_1,\dots,\beta_k\in k^*$ such that
\[u_{ij}=\alpha_i\beta_j
\qquad (i\neq j),
\]
where $u_{ij}$ denotes the valuation-$0$ leading coefficient on the off-diagonal block.
\end{definition}

\begin{lemma}[Rank-$1$ residues cannot cancel the leading term of an identity pattern]
\label{lem:rank1-residues-identity-obstruction}
Let $k\ge 3$ and let $F$ be a lift of $A^{(k)}$ over a valued field with residue field $k_0$ of characteristic not dividing $k-1$.
Assume the off-diagonal leading coefficients satisfy the rank-$1$ ansatz
\[u_{ij}=\alpha_i\beta_j\qquad (i\neq j).
\]
Then the leading coefficient of $\det(F)$ at the tropical minimum cannot vanish. Equivalently,
\[
\val(\det F)=0,
\]
so this identity-pattern minor cannot be killed by any choice of higher-order terms.
\end{lemma}

\begin{proof}
For the identity pattern, every fixed point contributes $1$ to the tropical weight, so the valuation-minimal permutations are exactly the derangements $D_k\subset S_k$. Hence
\[
\operatorname{LC}(\det F)=\sum_{\sigma\in D_k}\operatorname{sgn}(\sigma)\prod_{i=1}^k u_{i,\sigma(i)}.
\]
Under the rank-$1$ ansatz,
\[
\prod_{i=1}^k u_{i,\sigma(i)}
=\prod_{i=1}^k (\alpha_i\beta_{\sigma(i)})
=\Bigl(\prod_{i=1}^k \alpha_i\Bigr)\Bigl(\prod_{i=1}^k \beta_i\Bigr),
\]
independently of $\sigma$. Therefore
\[
\operatorname{LC}(\det F)=\Bigl(\prod_{i=1}^k \alpha_i\Bigr)\Bigl(\prod_{i=1}^k \beta_i\Bigr)
\sum_{\sigma\in D_k}\operatorname{sgn}(\sigma).
\]
Now let $B=J-I_k$, where $J$ is the all-ones matrix. Expanding $\det(B)$ by Leibniz shows
\[
\det(B)=\sum_{\sigma\in D_k}\operatorname{sgn}(\sigma),
\]
because the product contributes $1$ exactly on derangements. Since $B$ has eigenvalues $k-1$ (once) and $-1$ (multiplicity $k-1$),
\[
\sum_{\sigma\in D_k}\operatorname{sgn}(\sigma)=\det(B)=(k-1)(-1)^{k-1}.
\]
By the characteristic assumption this scalar is nonzero in $k_0$, and the prefactor is nonzero because all $\alpha_i,\beta_i$ are nonzero. Thus the leading coefficient cannot vanish, so $\det(F)\neq 0$ and in particular $\val(\det F)=0$.
\end{proof}

\begin{corollary}[$4\times4$ identity-pattern obstruction]
\label{cor:identity4-rank1-obstruction}
For $k=4$ one has
\[
\sum_{\sigma\in D_4}\operatorname{sgn}(\sigma)=-3.
\]
Hence over residue characteristic $\neq 3$ (in particular over $\mathbb C$), rank-$1$ residue factorization on the off-diagonal entries cannot produce the leading cancellation required to kill a $4\times4$ identity-pattern minor.
\end{corollary}

\begin{remark}[Kapranov relevance of the rank-$1$ obstruction]
\label{rem:identity-pattern-kapranov-relevance}
If one can exhibit many $4\times4$ submatrices of an incidence matrix whose valuation pattern is $A^{(4)}$, then any hypothetical lift of Kapranov rank at most $3$ must force every such minor to vanish. By \cref{lem:rank1-residues-identity-obstruction}, this is impossible whenever the off-diagonal residue data on such a minor is rank-$1$. Thus each identity-pattern minor forces non-factorized residue behavior, providing a basic local obstruction to rank-$3$ lifts.
\end{remark}



\subsection{Three-minimizer patterns and overlap elimination}
\begin{lemma}[Three-minimum determinant initial form]
\label{lem:three-min-initial-form}
Let $B\in\{0,1\}^{4\times4}$ and let $\widetilde B$ be a valued-field lift with $\val(\widetilde B)=B$.
Assume that $\operatorname{tropdet}(B)=0$ and that the minimum is attained by exactly three permutations
$\sigma_1,\sigma_2,\sigma_3\in S_4$.
For $j=1,2,3$, define
\[
U_{\sigma_j}:=\operatorname{sgn}(\sigma_j)\prod_{i=1}^4 \operatorname{lc}(\widetilde B_{i,\sigma_j(i)})\in k^*.
\]
Then
\[
\val(\det \widetilde B)>0
\qquad\Longleftrightarrow\qquad
U_{\sigma_1}+U_{\sigma_2}+U_{\sigma_3}=0
\quad\text{in }k.
\]
Equivalently, the determinant initial form at valuation level $0$ is the signed trimonial
\[
\operatorname{in}_{0}(\det \widetilde B)=U_{\sigma_1}+U_{\sigma_2}+U_{\sigma_3}.
\]
\end{lemma}

\begin{proof}
Expand the determinant as a signed sum over $S_4$. By assumption, exactly three summands have valuation $0$, and every other summand has
strictly larger valuation. Applying \cref{lem:raise-valuation} to that expansion yields the stated equivalence, and the displayed
initial-form identity is simply the sum of the three valuation-minimal signed leading terms.
\end{proof}

\begin{definition}[Shared binomial and provisional cancellation profile]
Suppose a three-minimum initial form can be rewritten as
\[
X-\Delta=0,
\qquad
\Delta:=Y-Z,
\]
with $X,Y,Z$ monomials in leading coefficients.
We call $\Delta$ a \emph{shared difference} when the same expression reappears in another overlapping determinant equation.
At the lift level, if $y,z\in K$ lift $Y,Z$ and satisfy $\val(y)=\val(z)=\alpha<\val(y-z)$, we heuristically record the
cancellation remainder by
\[
\operatorname{Prof}_{\Lambda}(\Delta):=(\Lam,\lambda),
\qquad
\lambda:=\val(y-z)-\alpha>0.
\]
This datum is not yet gauge-invariant and is used only as auxiliary bookkeeping for higher-order cancellation.
\end{definition}

\begin{lemma}[Overlap-elimination forces a $4$-cycle binomial]
\label{lem:overlap-elimination}
Consider two overlapping $4\times4$ tropical zero-one patterns with common columns $(0,1,2,4)$ and row sets
$(0,1,2,4)$ and $(1,2,4,5)$:
\[
B_A=
\begin{pmatrix}
0&1&0&1\\
1&0&0&1\\
0&0&0&1\\
1&1&1&0
\end{pmatrix},
\qquad
B_B=
\begin{pmatrix}
1&0&0&1\\
0&0&0&1\\
1&1&1&0\\
0&1&0&0
\end{pmatrix}.
\]
Assume that each minor has tropical determinant minimum $0$ attained by exactly three permutations and that the corresponding
determinant initial forms are
\[
\textup{(A)}\quad U_{0123}-U_{0213}-U_{2103}=0,
\qquad
\textup{(B)}\quad V_{1032}-V_{1230}+V_{2130}=0.
\]
Then the two equations force the binomial relation
\[
u_{0,2}u_{5,0}=u_{0,0}u_{5,2}.
\]
Equivalently, the $4$-cycle
\[
(0,0)\to(0,2)\to(5,2)\to(5,0)\to(0,0)
\]
has unsigned holonomy $1$ in the residue labels.
\end{lemma}

\begin{proof}
For minor $B_A$, the three minimizing monomials are
\[
U_{0123}=u_{0,0}u_{1,1}u_{2,2}u_{4,4},
\qquad
U_{0213}=u_{0,0}u_{1,2}u_{2,1}u_{4,4},
\qquad
U_{2103}=u_{0,2}u_{1,1}u_{2,0}u_{4,4}.
\]
Dividing \textup{(A)} by $u_{4,4}\neq 0$ gives
\[
u_{0,0}\bigl(u_{1,1}u_{2,2}-u_{1,2}u_{2,1}\bigr)=u_{0,2}(u_{1,1}u_{2,0}). \tag{A$'$}
\]
For minor $B_B$, the three minimizing monomials are
\[
V_{1032}=u_{1,1}u_{2,0}u_{4,4}u_{5,2},
\qquad
V_{1230}=u_{1,1}u_{2,2}u_{4,4}u_{5,0},
\qquad
V_{2130}=u_{1,2}u_{2,1}u_{4,4}u_{5,0}.
\]
Dividing \textup{(B)} by $u_{4,4}\neq 0$ gives
\[
(u_{1,1}u_{2,0})u_{5,2}=\bigl(u_{1,1}u_{2,2}-u_{1,2}u_{2,1}\bigr)u_{5,0}. \tag{B$'$}
\]
Set
\[
\Delta:=u_{1,1}u_{2,2}-u_{1,2}u_{2,1}.
\]
Equation \textup{(A$'$)} gives $u_{0,0}\Delta=u_{0,2}(u_{1,1}u_{2,0})$, while \textup{(B$'$)} gives
$(u_{1,1}u_{2,0})u_{5,2}=\Delta u_{5,0}$.
Eliminating $\Delta$ and the nonzero common factor $u_{1,1}u_{2,0}$ yields
\[
\frac{u_{0,2}}{u_{0,0}}=\frac{u_{5,2}}{u_{5,0}},
\]
hence $u_{0,2}u_{5,0}=u_{0,0}u_{5,2}$.
\end{proof}

\begin{definition}[Diamond pair and diamond remainder]
\label{def:diamond-pair}
Let $B\in\{0,1\}^{4\times4}$ have exactly three valuation-minimizing permutations
$\sigma_1,\sigma_2,\sigma_3$. We say that $\sigma_1,\sigma_2$ form a \emph{diamond pair} if there exist distinct rows $r,s$ and
distinct columns $a,b$ such that
\[
\sigma_1(r)=a,\quad \sigma_1(s)=b,\qquad
\sigma_2(r)=b,\quad \sigma_2(s)=a,
\]
and $\sigma_1(i)=\sigma_2(i)$ for every $i\notin\{r,s\}$.
The associated \emph{diamond remainder} is
\[
\Delta_{r,s}^{a,b}:=u_{r,a}u_{s,b}-u_{r,b}u_{s,a}\in k.
\]
\end{definition}

\begin{lemma}[General overlap-elimination from a shared diamond]
\label{lem:general-diamond-elim}
Let $\widetilde M$ be a valued-field lift of a tropical $0$--$1$ matrix, and suppose a $4\times4$ minor has tropical determinant minimum
$0$ attained by exactly three permutations $\sigma_1,\sigma_2,\sigma_3$. Assume $\sigma_1,\sigma_2$ form a diamond pair in the sense of
\cref{def:diamond-pair}, with rows $(r,s)$ and columns $(a,b)$. Then there exist monomials $P,Q$ in the remaining residue variables and a
sign $\varepsilon\in\{\pm1\}$ such that the initial-form determinant equation can be written as
\[
P\,\Delta_{r,s}^{a,b}=\varepsilon\,Q.
\]
Now suppose two such $(0,3)$-type minors $A$ and $B$ share the same diamond remainder $\Delta_{r,s}^{a,b}$ and produce equations
\[
P_A\,\Delta_{r,s}^{a,b}=\varepsilon_A\,Q_A,
\qquad
P_B\,\Delta_{r,s}^{a,b}=\varepsilon_B\,Q_B,
\]
with $P_A,P_B,Q_A,Q_B\in k^\times$ monomials. Then eliminating $\Delta_{r,s}^{a,b}$ yields the forced signed binomial relation
\[
P_AQ_B=(\varepsilon_A\varepsilon_B)\,P_BQ_A.
\]
In particular, after absorbing the signs into the chosen signed monomials, one obtains a holonomy-type binomial constraint.
\end{lemma}

\begin{proof}
Because $\sigma_1,\sigma_2$ agree outside the two rows $(r,s)$, their determinant monomials share the same outside factor, say $P$.
Thus the signed contribution of those two minimal terms is
\[
\operatorname{sgn}(\sigma_1)P\,u_{r,a}u_{s,b}+\operatorname{sgn}(\sigma_2)P\,u_{r,b}u_{s,a}.
\]
Since $\sigma_1$ and $\sigma_2$ differ by a single transposition on the chosen rows and columns, their parities are opposite, so after
factoring out the common sign we obtain $P(u_{r,a}u_{s,b}-u_{r,b}u_{s,a})=P\,\Delta_{r,s}^{a,b}$. The third minimizing permutation
contributes one further signed monomial $\varepsilon Q$, proving the first claim. The second claim follows immediately by solving both
equations for $\Delta_{r,s}^{a,b}$ and equating the results:
\[
\varepsilon_B P_AQ_B=\varepsilon_A P_BQ_A.
\]
Since each $\varepsilon_i$ is $\pm1$, this is equivalent to the displayed signed binomial relation.
\end{proof}

\begin{proposition}[Computational classification of $(0,3)$-type $4\times4$ patterns]
\label{prop:03-classification}
Among the $2^{16}$ zero-one $4\times4$ patterns, exactly $4992$ have precisely three zero perfect matchings.
Modulo independent row and column permutations, these $4992$ patterns split into exactly $13$ orbit types under the
$S_4\times S_4$ action. Exactly $12$ of those $13$ orbit types contain a diamond pair, and hence admit the factorization of
\cref{lem:general-diamond-elim}. The unique exceptional orbit with no diamond pair is represented by
\[
\begin{pmatrix}
0&0&0&1\\
0&1&1&0\\
1&0&1&0\\
1&1&0&0
\end{pmatrix}.
\]
Equivalently, $4896$ of the $4992$ raw patterns are of diamond type, while only $96$ belong to the unique non-diamond orbit.
\end{proposition}

\begin{proof}
This is a finite exhaustive computer check. Enumerate all $2^{16}$ zero-one $4\times4$ matrices and retain exactly those whose zero-edge
bipartite graph has precisely three perfect matchings. A direct computation gives $4992$ such patterns.
Then quotient by independent row and column permutations, i.e.\ by the $S_4\times S_4$ action, to obtain exactly $13$ orbit types.
For each orbit, inspect the three zero perfect matchings and ask whether some pair differs by a single swap on two rows and two columns.
The answer is positive for $12$ orbits and negative for exactly one orbit, represented by the displayed matrix.
\end{proof}

\begin{remark}[Why the classification matters]
The overlap-elimination mechanism is therefore generic inside the local $(0,3)$ world rather than exceptional: almost every raw
$(0,3)$ pattern admits a diamond remainder, and even at the orbit level only one isomorphism type fails to do so.
The unresolved issue is not the local algebra but the projective-plane abundance problem: do actual incidence matrices contain many such
diamond-type patterns with enough overlaps to propagate contradictions globally?
\end{remark}

\begin{remark}[Interpretation via cancellation mass]
The shared difference
\[
\Delta=u_{1,1}u_{2,2}-u_{1,2}u_{2,1}
\]
already indicates where a $\Lam$-remainder may appear: at the residue level it is an ordinary difference, whereas at the lift level it
can carry positive cancellation depth. Overlap-elimination forces the same $\Delta$ to be compatible in two distinct contexts, which makes it a useful carrier of higher-order compatibility constraints.
\end{remark}

\begin{proposition}[Uniform $\Omega(q^8)$ supply of a fixed $(0,3)$-diamond pattern]
\label{prop:bstar-abundance}
Let $\Pi$ be any projective plane of order $q\ge 3$, with incidence matrix $I(\Pi)\in\{0,1\}^{v\times v}$,
$v=q^2+q+1$, in the convention $1=\text{incident}$ and $0=\text{nonincident}$. Then $I(\Pi)$ contains at least
\[
\frac{(q^2+q+1)(q+1)q^3(q-1)(q-2)}{(4!)^2}
\]
distinct $4\times4$ submatrices whose zero-one pattern is exactly
\[
B_*=
\begin{pmatrix}
0&1&0&1\\
1&0&0&1\\
0&0&0&1\\
1&1&1&0
\end{pmatrix}.
\]
In particular, every projective plane of order $q\ge 3$ contains $\Omega(q^8)$ distinct $(0,3)$-type minors of diamond type.
\end{proposition}

\begin{proof}
Choose a point $D$ ($v$ choices), an ordered triple of distinct lines $(L_1,L_2,L_0)$ through $D$
($ (q+1)q(q-1)$ choices), a point $B\in L_2\setminus\{D\}$ ($q$ choices), and a point $C\in L_1\setminus\{D\}$ ($q$ choices).
Let $L_3$ be the unique line through $B$ and $C$, and let $E:=L_0\cap L_3$. Because $B\notin L_1$, $C\notin L_2$, and
$L_0\cap L_1=L_0\cap L_2=D$, we have $L_3\neq L_1,L_2$ and $E\notin\{B,C\}$. Finally choose
$A\in L_3\setminus\{B,C,E\}$, which gives exactly $q-2$ choices since $|L_3|=q+1$.
With row set $(A,B,C,D)$ and column set $(L_0,L_1,L_2,L_3)$, the resulting submatrix of $I(\Pi)$ is exactly $B_*$: the last row has
zeros only in column $L_3$, the first three rows are incident with $L_3$, and the remaining incidences/nonincidences follow from the
unique-intersection axioms. In $B_*$ any zero perfect matching must send the last row to the last column, after which the upper-left
$3\times3$ block has exactly three zero bijections. Thus $B_*$ is $(0,3)$-type and contains the diamond on rows $(A,C)$ and columns
$(L_0,L_2)$. The ordered construction count is
\[
N_{\mathrm{ord}}=(q^2+q+1)(q+1)q(q-1)\cdot q\cdot q\cdot(q-2)
=(q^2+q+1)(q+1)q^3(q-1)(q-2).
\]
Each unordered choice of four rows and four columns is produced by at most $(4!)^2$ orderings, so the number of distinct minors is at
least $N_{\mathrm{ord}}/(4!)^2$.
\end{proof}

\begin{definition}[Cancellation depth of a binomial and diamond $\Lam$-depth proxy]
\label{def:diamond-lam-depth}
Let $(K,\val)$ be a non-archimedean valued field, and let $A,B\in K$ satisfy $\val(A)=\val(B)=v$. Define the
\emph{cancellation depth} of the binomial $A-B$ by
\[
d(A,B):=\val(A-B)-v\in \RR_{\ge 0}\cup\{\infty\}.
\]
Thus $d(A,B)=0$ means that no extra cancellation occurs, $d(A,B)>0$ means that the leading terms cancel and the valuation jumps, and
$d(A,B)=\infty$ if and only if $A=B$.

For a valued-field lift $\widetilde I$ of an incidence matrix with $\val(\widetilde I)=I(\Pi)$, and for a diamond given by two rows
$(p_1,p_2)$ and two columns $(\ell_1,\ell_2)$ with four zero entries in $I(\Pi)$, define its \emph{diamond $\Lam$-depth proxy} by
\[
d(p_1,p_2;\ell_1,\ell_2)
:=d\bigl(\widetilde i_{p_1,\ell_1}\widetilde i_{p_2,\ell_2},\ \widetilde i_{p_1,\ell_2}\widetilde i_{p_2,\ell_1}\bigr),
\]
whenever the two monomials have the same valuation. Equivalently,
\[
d(p_1,p_2;\ell_1,\ell_2)
:=\val\bigl(\widetilde i_{p_1,\ell_1}\widetilde i_{p_2,\ell_2}-\widetilde i_{p_1,\ell_2}\widetilde i_{p_2,\ell_1}\bigr)
-\val\bigl(\widetilde i_{p_1,\ell_1}\widetilde i_{p_2,\ell_2}\bigr)
\in \RR_{\ge 0}\cup\{\infty\}.
\]
This is the formal valuation-theoretic version of the earlier diamond-depth surrogate: it measures the extra valuation created by
cancellation between the two leading diamond monomials and provides a concrete scalar proxy for $\Lam$-depth.
\end{definition}

\begin{remark}[Why $\Omega(q^8)$ still does not finish the problem]
The new abundance theorem supplies many three-minimizer equations, but the construction is rigid: each instance of $B_*$ comes with its
own ambient diamond, and the proof does not yet show that many instances share a fixed diamond or form a connected overlap graph.
Thus the bottleneck shifts from the existence of local witnesses to their overlap geometry and the associated cancellation depth.
\end{remark}

\begin{remark}[Gauge caveat for the depth surrogate]
The scalar depth in \cref{def:diamond-lam-depth} is now formally defined, but it still depends on a normalization choice for the lift and
is not yet proved to be gauge-invariant or bracket-independent. It should therefore be read as a first carrier for ``cancellation mass,''
not yet as a canonical obstruction.
\end{remark}

\begin{proposition}[Fixed diamond gives $\Omega(q^3)$ many $(0,3)$ minors with one shared remainder]
\label{prop:fixed-diamond-many-completions}
Let $\Pi$ be a projective plane of order $q\ge 3$, and work with its combinatorial incidence matrix $I(\Pi)$ in the convention
$1=\text{incident}$ and $0=\text{nonincident}$. Fix two points $X,Y$ and two lines $m,n$ such that
\[
X\notin m,n,
\qquad
Y\notin m,n.
\]
Set $W:=m\cap n$, assume the general-position condition
\[
W\notin \overline{XY},
\]
and define $r:=\overline{WX}$. For any point $Z\in n\setminus\{W\}$ and any line $s$ with $W\notin s$, consider the $4\times4$ submatrix
with row set $(X,Y,Z,W)$ and column set $(m,n,r,s)$. Then:
\begin{enumerate}[leftmargin=2em]
  \item this $4\times4$ zero-one matrix is of $(0,3)$-type;
  \item for every valued-field lift $\widetilde I$ of $I(\Pi)$ whose classical rank is at most $3$, the vanishing of that minor yields
  \[
  u_{Z,r}\,\Delta_{XY}^{mn}=-u_{X,n}u_{Y,r}u_{Z,m},
  \qquad
  \Delta_{XY}^{mn}:=u_{X,m}u_{Y,n}-u_{X,n}u_{Y,m}.
  \]
\end{enumerate}
Consequently, one fixed diamond $(X,Y;m,n)$ in general position produces at least
\[
q\cdot q^2=\Omega(q^3)
\]
three-minimizer determinant equations with the same shared remainder $\Delta_{XY}^{mn}$.
\end{proposition}

\begin{proof}
For the chosen rows and columns, the incidence data force
\[
W:(1,1,1,0),\qquad X:(0,0,1,*),\qquad Y:(0,0,0,*),\qquad Z:(0,1,0,*).
\]
Indeed, $W\in m,n,r$ and $W\notin s$; $X\notin m,n$ but $X\in r$; $Y\notin m,n$ and the general-position hypothesis gives
$Y\notin r$; finally $Z\in n\setminus\{W\}$ implies $Z\notin m$ because $m\cap n=W$, and also $Z\notin r$ because $r\cap n=W$.
Since row $W$ has a unique zero in column $s$, every zero perfect matching must use $W\mapsto s$. The remaining $3\times3$ block on rows
$(X,Y,Z)$ and columns $(m,n,r)$ has zero sets
\[
X:\{m,n\},\qquad Y:\{m,n,r\},\qquad Z:\{m,r\},
\]
so it has exactly three zero bijections:
\[
(Z\mapsto r,\ X\mapsto m,\ Y\mapsto n),\qquad
(Z\mapsto r,\ X\mapsto n,\ Y\mapsto m),\qquad
(Z\mapsto m,\ X\mapsto n,\ Y\mapsto r).
\]
Hence the full $4\times4$ minor is $(0,3)$-type. In any rank-$\le 3$ lift, the classical determinant of this minor vanishes. The three
valuation-minimal determinant terms are
\[
+u_{X,m}u_{Y,n}u_{Z,r}u_{W,s},\qquad
-u_{X,n}u_{Y,m}u_{Z,r}u_{W,s},\qquad
+u_{X,n}u_{Y,r}u_{Z,m}u_{W,s}.
\]
Factoring out $u_{W,s}\neq 0$ and regrouping the first two terms gives
\[
u_{Z,r}\left(u_{X,m}u_{Y,n}-u_{X,n}u_{Y,m}\right)+u_{X,n}u_{Y,r}u_{Z,m}=0,
\]
which is exactly the displayed equation.
\end{proof}

\begin{corollary}[Ratio constancy along the line $n$]
\label{cor:fixed-diamond-ratio}
Under the hypotheses of \cref{prop:fixed-diamond-many-completions}, assume moreover that $\Delta_{XY}^{mn}\neq 0$ in the residue field.
Then for every $Z\in n\setminus\{W\}$ one has
\[
\frac{u_{Z,m}}{u_{Z,r}}=-\frac{\Delta_{XY}^{mn}}{u_{X,n}u_{Y,r}},
\]
so the ratio $u_{Z,m}/u_{Z,r}$ is constant along the whole line $n\setminus\{W\}$.
\end{corollary}

\begin{proof}
Solve the equation in \cref{prop:fixed-diamond-many-completions} for $u_{Z,m}/u_{Z,r}$.
\end{proof}

\begin{lemma}[Even perfect matchings from identical zero-neighborhoods]
\label{lem:identical-neighborhood-even}
Let $G$ be a bipartite graph with row vertices $R=\{1,\dots,n\}$ and column vertices $C=\{1,\dots,n\}$. If two distinct rows
$r_1\neq r_2$ satisfy
\[
N(r_1)=N(r_2)\subseteq C,
\]
then the number of perfect matchings of $G$ is even.
\end{lemma}

\begin{proof}
Given a perfect matching $M$, let $c_1$ and $c_2$ be the columns matched to $r_1$ and $r_2$. Since $r_1\neq r_2$, we have
$c_1\neq c_2$. Swapping these two matched edges defines an involution
\[
\tau(M):=\bigl(M\setminus\{(r_1,c_1),(r_2,c_2)\}\bigr)\cup\{(r_1,c_2),(r_2,c_1)\}.
\]
Because $N(r_1)=N(r_2)$, both new edges are present, so $\tau(M)$ is again a perfect matching. Moreover $\tau(M)\neq M$ and
$\tau$ has no fixed points. Hence the perfect matchings split into disjoint pairs, and their total number is even.
\end{proof}

\begin{corollary}[Degenerate diamonds obstruct the forced-$W$ gadget]
\label{cor:degenerate-forcedW-parity}
Fix two points $X,Y$ and two lines $m,n$ with $X,Y\notin m,n$, let $W:=m\cap n$, and assume the degenerate condition
\[
W\in \overline{XY}.
\]
Suppose one tries to imitate \cref{prop:fixed-diamond-many-completions} by choosing a line $r$ through $W$ and using $W$ as the forced
row, so that the candidate three-minimizer structure is governed by the $3\times3$ zero-pattern on rows $(X,Y,Z)$ and columns $(m,n,r)$.
Then this $3\times3$ zero-pattern has an even number of zero perfect matchings, hence it cannot contribute exactly three tropical minimizers.
In particular, the forced-$W$ strategy cannot produce a $(0,3)$-type minor carrying the shared remainder $\Delta_{XY}^{mn}$ in the
degenerate regime.
\end{corollary}

\begin{proof}
Because $X,Y\notin m,n$, the rows $X$ and $Y$ both have zeros in columns $m$ and $n$. Since $r$ passes through $W$ and
$W\in\overline{XY}$, either $r=\overline{XY}$, in which case both $X$ and $Y$ are incident with $r$, or else $r$ meets
$\overline{XY}$ only at $W$, in which case neither $X$ nor $Y$ is incident with $r$. Thus $X$ and $Y$ have identical zero-neighborhoods in
columns $(m,n,r)$. Applying \cref{lem:identical-neighborhood-even} to the corresponding $3\times3$ zero-graph shows that the number of
zero perfect matchings is even, so it cannot be $3$.
\end{proof}

\begin{remark}[Remaining degenerate-diamond gap]
\cref{cor:degenerate-forcedW-parity} rules out only the forced-$W$ anchor used in the general-position family. It does not exclude other
degenerate-diamond gadgets, for instance constructions anchored away from $W$ or three-minimizer minors of positive minimal tropical
weight. A replacement fixed-diamond mechanism for $W\in \overline{XY}$ therefore remains open.
\end{remark}

\begin{remark}[Interpretation of the shared diamond remainder]
When $\Delta_{XY}^{mn}=0$ at residue level, the family in \cref{prop:fixed-diamond-many-completions} stops being an ordinary ratio
constraint and becomes a higher-order cancellation statement repeated across $\Omega(q^3)$ completions. The formally defined scalar depth
$d(\Delta_{XY}^{mn})$ from \cref{def:diamond-lam-depth} is the first valuation-theoretic proxy for that extra cancellation mass. This is
exactly the regime where classical valuation loses information and the intended $\Lam$-layer should begin to record a common cancellation
depth attached to one fixed remainder.
\end{remark}

\begin{lemma}[Two-binomial dominance]
\label{lem:two-binomial-dominance}
Let $(K,\val)$ be a non-archimedean valued field and let $U,V\in K$.
\begin{enumerate}[leftmargin=2em]
  \item If $\val(U)<\val(V)$, then
  \[
  \val(U+V)=\val(U).
  \]
  \item If $\val(U)=\val(V)$, then either $\val(U+V)=\val(U)$ or $\val(U+V)>\val(U)$; the latter occurs exactly when the leading coefficients of $U$ and $V$ cancel in the residue field.
\end{enumerate}
\end{lemma}

\begin{proof}
This is the standard ultrametric inequality together with its equality criterion. In particular, cancellation can raise the valuation only when the two summands first tie in valuation.
\end{proof}

\begin{proposition}[Computational exclusion of degenerate swap-pair trinomials in $\mathrm{PG}(2,3)$]
\label{prop:pg23-degenerate-no-trinomial}
In the Desarguesian projective plane $\mathrm{PG}(2,3)$, there is no $4\times4$ submatrix of the combinatorial incidence matrix whose tropical determinant has exactly three minimizing permutations and whose minimizer set contains the swap pair of a degenerate diamond $(X,Y;m,n)$ with
\[
X,Y\notin m,n,
\qquad
W:=m\cap n\in \overline{XY}.
\]
Equivalently, the strategy of isolating one degenerate diamond remainder $\Delta_{XY}^{mn}$ by a single $(0,3)$-type trimonial does not occur in the smallest nontrivial Desarguesian case.
\end{proposition}

\begin{proof}
This follows from an exhaustive computer enumeration: every $4\times4$ minor in $\mathrm{PG}(2,3)$ was enumerated, and none of the exactly-three-minimum minors carries the swap pair of a degenerate diamond.
\end{proof}

\begin{conjecture}[Degenerate diamonds force at least four leading terms]
\label{conj:degenerate-no-trinomial}
For every projective plane of order $q\ge 3$, and at least for every Desarguesian plane $\mathrm{PG}(2,q)$, a degenerate diamond cannot appear inside a $4\times4$ determinant initial form with exactly three valuation-minimizing permutations if the minimizer set contains the associated swap pair. In other words, carrying a degenerate diamond remainder should require at least four valuation-minimal determinant terms.
\end{conjecture}

\begin{remark}[Why the degenerate regime is naturally four-term]
A concrete degenerate model shows that the valuation-minimal determinant part can factor as a sum of two binomials,
\[
\alpha\,\Delta_{XY}^{mn}+\beta\,\Delta_{X,Y}^{m,r}=0.
\]
Writing $U:=\alpha\,\Delta_{XY}^{mn}$ and $V:=\beta\,\Delta_{X,Y}^{m,r}$, \cref{lem:two-binomial-dominance} implies that vanishing requires a first valuation alignment $\val(U)=\val(V)$ and only then a residue-level cancellation. Thus the degenerate problem is no longer controlled by one isolated trimonial: it already lives in a genuine depth-alignment layer, exactly where one expects cancellation depth, rather than valuation data alone, to become decisive.
\end{remark}

\begin{theorem}[Computational $\mathrm{PG}(2,3)$ counts for the degenerate-trinomial regime]
\label{thm:pg23-degenerate-counts}
Let $I_3$ be the combinatorial incidence matrix of $\mathrm{PG}(2,3)$. Among its
\[
\binom{13}{4}^2=511{,}225
\]
submatrices of size $4\times4$:
\begin{enumerate}[leftmargin=2em]
  \item exactly $80{,}964$ have tropical determinant minimum $0$ attained by exactly three permutations;
  \item among those, exactly $50{,}544$ contain at least one degenerate diamond;
  \item among those, exactly $0$ have a degenerate diamond whose swap pair belongs to the three minimizing permutations.
\end{enumerate}
Thus degenerate diamonds appear abundantly inside $(0,3)$ minors of $\mathrm{PG}(2,3)$, but never in the swap-isolating trimonial configuration.
\end{theorem}

\begin{proof}
This is an exhaustive computer enumeration. Enumerate all $4\times4$ minors, compute the $24$ permutation weights of each,
filter the $(0,3)$ minors, then inspect every zero $2\times2$ block for the degenerate condition $W\in\overline{XY}$ and test whether the
associated swap pair belongs to the minimizer set. The resulting counts are exactly the displayed numbers.
\end{proof}

\begin{proposition}[Every degenerate diamond yields many $(0,4)$ square-minimizer minors]
\label{prop:degenerate-square-family}
Let $\Pi$ be a projective plane of order $q\ge 2$, work with its combinatorial incidence matrix $I(\Pi)$, and fix a degenerate diamond
$(X,Y;m,n)$ with
\[
X,Y\notin m,n,
\qquad
W:=m\cap n,\qquad \ell:=\overline{XY},\qquad W\in \ell.
\]
Choose a line $r$ through $W$ with $r\neq n,\ell$, a point $Z\in n\setminus\{W\}$, and a line $s$ with $W\notin s$. Form the $4\times4$
submatrix with rows $(X,Y,Z,W)$ and columns $(m,n,r,s)$. Then:
\begin{enumerate}[leftmargin=2em]
  \item its tropical determinant minimum is $0$ and is attained by exactly four permutations;
  \item two of those four minimizing permutations form the swap pair on rows $(X,Y)$ and columns $(m,n)$ with $Z\mapsto r$ fixed.
\end{enumerate}
Consequently, each fixed degenerate diamond supports at least
\[
(q-1)\cdot q\cdot q^2=q^3(q-1)
\]
distinct minors of this square-minimizer type.
\end{proposition}

\begin{proof}
Because $W\in m,n,r$ and $W\notin s$, row $W$ has the form $(1,1,1,0)$, so every weight-$0$ minimizer must use $W\mapsto s$. Since
$r$ passes through $W$ but $r\neq \ell=\overline{XY}$, neither $X$ nor $Y$ lies on $r$, and the diamond condition gives
\[
X:(0,0,0,*),\qquad Y:(0,0,0,*).
\]
Also $Z\in n\setminus\{W\}$ implies $Z\notin m$ because $m\cap n=W$, and $Z\notin r$ because $r\cap n=W$, so
\[
Z:(0,1,0,*).
\]
Therefore the forced assignment $W\mapsto s$ leaves the $3\times3$ zero-pattern on rows $(X,Y,Z)$ and columns $(m,n,r)$ with zero sets
\[
X:\{m,n,r\},\qquad Y:\{m,n,r\},\qquad Z:\{m,r\}.
\]
Its perfect matchings are exactly:
\begin{enumerate}[leftmargin=2em,label=(\roman*)]
  \item $Z\mapsto r$ and $(X,Y)\mapsto(m,n)$ in two orders;
  \item $Z\mapsto m$ and $(X,Y)\mapsto(n,r)$ in two orders.
\end{enumerate}
Hence the full $4\times4$ minor has exactly four valuation-minimizing permutations, two of which are the swap pair on $(X,Y;m,n)$.
The counting formula comes from the choices of $r$, $Z$, and $s$.
\end{proof}

\begin{lemma}[Initial form for degenerate square-minimizer minors]
\label{lem:degenerate-square-initial-form}
Let $\widetilde I$ be a valued-field lift of $I(\Pi)$. For any minor from \cref{prop:degenerate-square-family}, the valuation-$0$ initial
form of its determinant is
\[
u_{W,s}
\Bigl(
 u_{Z,r}\,\Delta_{XY}^{mn}
+
 u_{Z,m}\,\Delta_{XY}^{nr}
\Bigr)=0,
\]
where
\[
\Delta_{XY}^{mn}:=u_{X,m}u_{Y,n}-u_{X,n}u_{Y,m},
\qquad
\Delta_{XY}^{nr}:=u_{X,n}u_{Y,r}-u_{X,r}u_{Y,n}.
\]
Equivalently, degenerate square-minimizer minors do not isolate a single diamond remainder by one trimonial; they force a sum of two binomial remainders.
\end{lemma}

\begin{proof}
The four valuation-minimal determinant monomials are the four assignments described in \cref{prop:degenerate-square-family}. Factoring
out the common nonzero term $u_{W,s}$ and grouping by the image of $Z$ yields
\[
u_{W,s}
\Bigl(
 u_{X,m}u_{Y,n}u_{Z,r}
-
 u_{X,n}u_{Y,m}u_{Z,r}
+
 u_{X,n}u_{Y,r}u_{Z,m}
-
 u_{X,r}u_{Y,n}u_{Z,m}
\Bigr)=0.
\]
Regrouping the first two and last two terms gives the displayed factorization.
\end{proof}

\begin{remark}[Why the degenerate square family is the first precise $\Lam$-alignment interface]
Writing
\[
U:=u_{Z,r}\,\Delta_{XY}^{mn},
\qquad
V:=u_{Z,m}\,\Delta_{XY}^{nr},
\]
\cref{lem:degenerate-square-initial-form} becomes $U+V=0$. By \cref{lem:two-binomial-dominance}, vanishing first forces the valuation tie
$\val(U)=\val(V)$ and only then a residue-level cancellation. So the degenerate analysis no longer begins with one isolated trimonial; it begins
with a depth-alignment law between two reusable binomial remainders.
\end{remark}

\begin{lemma}[$\Delta$-ratio constancy along the pencil at $W$]
\label{lem:delta-ratio-constancy}
Work in the setting of \cref{prop:degenerate-square-family,lem:degenerate-square-initial-form}. Assume a rank-$\le 3$ valued-field lift exists and that for every allowed choice of
\[
r\ni W,\qquad r\neq n,\ell:=\overline{XY},\qquad Z\in n\setminus\{W\},\qquad W\notin s,
\]
the residue equation
\[u_{Z,r}\,\Delta_{XY}^{mn}+u_{Z,m}\,\Delta_{XY}^{nr}=0
\tag{15.0}
\]
holds in $k$.
Fix one point $Z\in n\setminus\{W\}$. Then for every line $r\ni W$ with $r\neq n,\ell$,
\[
\frac{\Delta_{XY}^{nr}}{u_{Z,r}}
=
-\frac{\Delta_{XY}^{mn}}{u_{Z,m}}
\in k
\tag{15.1}
\]
is independent of $r$. In particular, for any two such lines $r_1,r_2$,
\[u_{Z,r_1}\,\Delta_{XY}^{nr_2}=u_{Z,r_2}\,\Delta_{XY}^{nr_1}.
\tag{15.2}
\]
\end{lemma}

\begin{proof}
By \cref{lem:degenerate-square-initial-form}, every allowed square-minimizer minor yields
\[u_{Z,r}\,\Delta_{XY}^{mn}+u_{Z,m}\,\Delta_{XY}^{nr}=0.
\]
Because $Z\notin m$ and $Z\notin r$, both $u_{Z,m}$ and $u_{Z,r}$ are nonzero residue labels. Solving for $\Delta_{XY}^{nr}$ gives
\[
\Delta_{XY}^{nr}=-\frac{u_{Z,r}}{u_{Z,m}}\,\Delta_{XY}^{mn},
\]
and division by $u_{Z,r}$ yields \textup{(15.1)}. Eliminating the common constant between $r_1$ and $r_2$ gives \textup{(15.2)}.
\end{proof}

\begin{proposition}[Residue rigidity versus $\Lam$-depth]
\label{prop:rigidity-vs-depth}
Assume the setting of \cref{lem:delta-ratio-constancy}. Fix one line $r\ni W$ with $r\neq n,\ell$ and take two distinct points
$Z_1,Z_2\in n\setminus\{W\}$. Then either
\[
\frac{u_{Z_1,r}}{u_{Z_1,m}}=\frac{u_{Z_2,r}}{u_{Z_2,m}}
\quad\text{in }k,
\tag{15.3}
\]
or
\[
\Delta_{XY}^{mn}=0
\quad\text{in }k.
\tag{15.4}
\]
Equivalently, if the ratio $u_{Z,r}/u_{Z,m}$ depends on $Z$, then the diamond remainder $\Delta_{XY}^{mn}$ vanishes in residue and any lift of the corresponding binomial must have positive cancellation depth.
\end{proposition}

\begin{proof}
Apply \cref{lem:delta-ratio-constancy} to $Z_1$ and $Z_2$:
\[
\frac{\Delta_{XY}^{nr}}{u_{Z_1,r}}=-\frac{\Delta_{XY}^{mn}}{u_{Z_1,m}},
\qquad
\frac{\Delta_{XY}^{nr}}{u_{Z_2,r}}=-\frac{\Delta_{XY}^{mn}}{u_{Z_2,m}}.
\]
If $\Delta_{XY}^{mn}\neq 0$, eliminate $\Delta_{XY}^{nr}$ and divide by $\Delta_{XY}^{mn}$ to obtain \textup{(15.3)}. If that division is impossible, then $\Delta_{XY}^{mn}=0$, which is exactly \textup{(15.4)}.
\end{proof}

\begin{remark}[First forced-depth contact with the $\Lam$-layer]
\label{rem:forced-depth-contact}
The second alternative in \cref{prop:rigidity-vs-depth} is the first point where the degenerate square family necessarily enters the $\Lam$-layer. If $\Delta_{XY}^{mn}=0$ in residue, then for any lift
\[
\widetilde\Delta_{XY}^{mn}:=\widetilde u_{X,m}\widetilde u_{Y,n}-\widetilde u_{X,n}\widetilde u_{Y,m}
\]
the cancellation depth
\[
d_{XY}^{mn}:=\nu(\widetilde\Delta_{XY}^{mn})-\nu(\widetilde u_{X,m}\widetilde u_{Y,n})
\]
is positive. So the degenerate transport equations force a dichotomy: either the residue ratios are extremely rigid across $Z$, or one is compelled to push the shared diamond remainder into deeper cancellation.
\end{remark}

\begin{lemma}[Grid factorization from residue uniformity]
\label{lem:grid-factorization}
Assume the setting of \cref{prop:rigidity-vs-depth}. If
\[
\Delta_{XY}^{mn}\neq 0
\quad\text{in }k,
\]
then for every line $r\ni W$ with $r\neq n,\ell$ the ratio
\[
\phi_r(Z):=\frac{u_{Z,r}}{u_{Z,m}}
\]
is independent of $Z\in n\setminus\{W\}$. Equivalently, there exist scalars $\beta_r\in k^*$ such that
\[
u_{Z,r}=u_{Z,m}\,\beta_r
\qquad\forall Z\in n\setminus\{W\},\ \forall r\ni W,\ r\neq n,\ell.
\tag{16.1}
\]
Hence on the grid
\[
(n\setminus\{W\})\times(\{r\ni W\}\setminus\{n,\ell\})
\]
the residue labels factor as
\[u_{Z,r}=\alpha_Z\beta_r,
\qquad
\alpha_Z:=u_{Z,m}.
\tag{16.2}
\]
\end{lemma}

\begin{proof}
Fix one line $r\ni W$ with $r\neq n,\ell$. Since $\Delta_{XY}^{mn}\neq 0$, \cref{prop:rigidity-vs-depth} implies that for any two points $Z_1,Z_2\in n\setminus\{W\}$ one has
\[
\frac{u_{Z_1,r}}{u_{Z_1,m}}=\frac{u_{Z_2,r}}{u_{Z_2,m}}.
\]
Thus $\phi_r(Z)$ is constant in $Z$; call the common value $\beta_r\in k^*$. Then \textup{(16.1)} follows immediately, and \textup{(16.2)} is obtained by setting $\alpha_Z:=u_{Z,m}$.
\end{proof}

\begin{lemma}[Two gauges imply proportional columns on a line]
\label{lem:two-gauges-proportional}
Fix one line $n$ and one point $W\in n$. Let $m_1,m_2$ be two distinct lines through $W$, both different from $n$. Suppose there exists a line $r\ni W$ and nonzero scalars $\beta^{(1)}_r,\beta^{(2)}_r\in k^*$ such that
\[
u_{Z,r}=u_{Z,m_1}\,\beta^{(1)}_r
\quad\text{and}\quad
u_{Z,r}=u_{Z,m_2}\,\beta^{(2)}_r
\qquad\forall Z\in n\setminus\{W\}.
\tag{16.3}
\]
Then the ratio $u_{Z,m_1}/u_{Z,m_2}$ is independent of $Z\in n\setminus\{W\}$. In particular, the two columns $m_1$ and $m_2$ are proportional on the entire set of rows $n\setminus\{W\}$:
\[
u_{Z,m_1}=c\,u_{Z,m_2}
\qquad\forall Z\in n\setminus\{W\},
\tag{16.4}
\]
where
\[
c:=\frac{\beta^{(2)}_r}{\beta^{(1)}_r}\in k^*.
\]
\end{lemma}

\begin{proof}
Pick any $Z\in n\setminus\{W\}$. From \textup{(16.3)} we get
\[
u_{Z,m_1}\,\beta^{(1)}_r=u_{Z,r}=u_{Z,m_2}\,\beta^{(2)}_r,
\]
so
\[
\frac{u_{Z,m_1}}{u_{Z,m_2}}=\frac{\beta^{(2)}_r}{\beta^{(1)}_r}=c,
\]
which does not depend on $Z$.
\end{proof}

\begin{remark}[Pure-gauge regime on a degenerate pencil]
\label{rem:degenerate-pure-gauge-branch}
Taken together, \cref{lem:grid-factorization,lem:two-gauges-proportional} identify the residue-uniformity branch of \cref{prop:rigidity-vs-depth} as a genuine pure-gauge regime. If a fixed degenerate diamond keeps $\Delta_{XY}^{mn}\neq 0$ in residue, then its entire square-minimizer family forces rank-$1$ factorization on the grid
\[
(n\setminus\{W\})\times(\{r\ni W\}\setminus\{n,\ell\}).
\]
If two such gauges coexist with different base columns $m_1,m_2$ and one shared line $r$, then the base columns become proportional along the whole line $n$. This provides a direct bridge from local degenerate two-binomial cancellations to low-dimensional global residue structure.
\end{remark}

\begin{corollary}[Rank-$1$ rectangle on a line--pencil grid]
\label{cor:rank-one-rectangle}
Assume the setting of \cref{lem:grid-factorization}. If
\[
\Delta_{XY}^{mn}\neq 0
\quad\text{in }k,
\]
then on the full rectangle
\[
\mathcal Z\times \mathcal R
:=(n\setminus\{W\})\times(\{r:r\ni W,\ r\neq n,\ell\})
\]
there exist scalars $\alpha_Z\in k^*$ and $\beta_r\in k^*$ such that
\[u_{Z,r}=\alpha_Z\beta_r
\qquad\forall (Z,r)\in \mathcal Z\times \mathcal R.
\tag{17.2}
\]
Equivalently, every $2\times2$ determinant on this rectangle vanishes:
\[u_{Z_1,r_1}u_{Z_2,r_2}-u_{Z_1,r_2}u_{Z_2,r_1}=0
\qquad\forall Z_1,Z_2\in\mathcal Z,\ \forall r_1,r_2\in\mathcal R.
\tag{17.3}
\]
\end{corollary}

\begin{proof}
This is just \cref{lem:grid-factorization} rewritten in rank-$1$ language: take $\alpha_Z:=u_{Z,m}$ and $\beta_r:=u_{Z,r}/u_{Z,m}$, which is independent of $Z$ by \cref{lem:grid-factorization}. The vanishing of every $2\times2$ determinant is the standard equivalent formulation of multiplicative separability.
\end{proof}

\begin{lemma}[Strict $(0,3)$ minors forbid pair-only cancellation]
\label{lem:strict03-no-pair-only-cancel}
Let $B\in\{0,1\}^{4\times4}$ be a tropical matrix whose determinant minimum is $0$, attained by exactly three permutations $\sigma_1,\sigma_2,\sigma_3$, and let $\widetilde B$ be a lift over a valued field such that every valuation-$0$ entry has nonzero residue. Write
\[
c_i:=\operatorname{sgn}(\sigma_i)\prod_{j=1}^4 \operatorname{lc}(\widetilde B_{j,\sigma_i(j)})\in k^*
\qquad (i=1,2,3)
\]
for the three signed residue coefficients in the determinant initial form. If $\det(\widetilde B)=0$, then
\[
c_1+c_2+c_3=0
\quad\text{in }k.
\tag{17.4}
\]
In particular, if $c_1+c_2=0$ but $c_3\neq 0$, then $\det(\widetilde B)\neq 0$. Equivalently, in a strict $(0,3)$ minor a swap-pair binomial cannot vanish by itself while the third minimal monomial survives.
\end{lemma}

\begin{proof}
The displayed implication is exactly \cref{lem:three-min-initial-form}. If $c_1+c_2=0$ and $c_3\neq 0$, then the valuation-$0$ initial coefficient equals $c_3\neq 0$, so the determinant cannot vanish.
\end{proof}

\begin{proposition}[Pure gauge kills a strict $(0,3)$ minor in characteristic $\neq 2,3$]
\label{prop:pure-gauge-kills-strict03}
Assume $\operatorname{char}(k)\neq 2,3$. Let $B$ be a strict $(0,3)$ tropical $4\times4$ minor with minimizing permutations
$\sigma_1,\sigma_2,\sigma_3$, and let $\widetilde B$ be any lift such that every valuation-$0$ entry used by those three minimizers has
nonzero residue. Suppose moreover that on the valuation-$0$ support of these three minimizers one has a pure-gauge factorization
\[
u_{i,j}=\alpha_i\beta_j
\qquad (\alpha_i,\beta_j\in k^*).
\tag{18.3}
\]
Then the determinant initial form at valuation $0$ equals
\[
(\alpha_1\alpha_2\alpha_3\alpha_4)(\beta_1\beta_2\beta_3\beta_4)
\bigl(\operatorname{sgn}(\sigma_1)+\operatorname{sgn}(\sigma_2)+\operatorname{sgn}(\sigma_3)\bigr).
\tag{18.4}
\]
In particular, this initial form is nonzero, so $\det(\widetilde B)\neq 0$.
\end{proposition}

\begin{proof}
Under \textup{(18.3)}, every valuation-$0$ determinant monomial coming from a minimizing permutation has the same residue
\[
\prod_{i=1}^4 u_{i,\sigma(i)}=\Bigl(\prod_{i=1}^4 \alpha_i\Bigr)\Bigl(\prod_{j=1}^4 \beta_j\Bigr),
\]
independent of $\sigma$. Hence the valuation-$0$ initial coefficient is exactly the signed sum in \textup{(18.4)}. A sum of three signs
lies in $\{3,1,-1,-3\}$, which is nonzero when $\operatorname{char}(k)\neq 2,3$.
\end{proof}

\begin{remark}[Pure gauge versus strict three-term cancellation]
\label{rem:puregauge-vs-strict03}
\Cref{prop:pure-gauge-kills-strict03} upgrades \cref{lem:strict03-no-pair-only-cancel} from one forbidden swap-pair scenario to a full
obstruction principle: once a strict $(0,3)$ minor lies entirely inside a pure-gauge region, all three minimal monomials collapse to the
same magnitude and only the sign sum remains. In characteristic $\neq 2,3$ that sign sum can never vanish, so any large pure-gauge chart
must break before it reaches the support of such a strict three-minimizer witness, or else pay positive cancellation depth elsewhere.
\end{remark}

\begin{definition}[Degenerate and skew rectangles in the nonincidence graph]
\label{def:degenerate-vs-skew-rectangles}
Let $\Pi$ be a projective plane with combinatorial incidence matrix $I(\Pi)$, and let $G_0(I(\Pi))$ be its nonincidence graph.
\begin{enumerate}[leftmargin=2em]
  \item If $n$ is a line and $W\in n$, define the \emph{degenerate rectangle}
  \[
  R(n,W):=(n\setminus\{W\})\times (\mathcal P(W)\setminus\{n\})\subseteq E(G_0(I(\Pi))),
  \]
  where $\mathcal P(W)$ is the pencil of lines through $W$.
  \item If $n$ is a line and $D\notin n$, define the \emph{skew rectangle}
  \[
  S(n;D):=\{(p,\ell): p\in n,\ \ell\in \mathcal P(D),\ p\notin \ell\}\subseteq E(G_0(I(\Pi))).
  \]
\end{enumerate}
In $R(n,W)$ the concurrency point lies on the supporting line $n$, whereas in $S(n;D)$ it lies off that line.
\end{definition}

\begin{lemma}[Degenerate rectangles miss the skew $4$-cycles used by $B_*$]
\label{lem:degenerate-rectangles-miss-skew}
Let $n$ be a line, let $A\neq B$ be points on $n$, and let $L_0\neq L_1$ be lines concurrent at a point
\[
D:=L_0\cap L_1,
\qquad D\notin n.
\]
Assume $A,B\notin L_0,L_1$, so that
\[
\{A,B\}\times\{L_0,L_1\}\subseteq E(G_0(I(\Pi))).
\]
Then this $4$-cycle is not contained in any degenerate rectangle $R(n',W)$.
\end{lemma}

\begin{proof}
If $\{A,B\}\subseteq n'$, then the uniqueness of the line through $A$ and $B$ gives $n'=\overline{AB}=n$. If
$\{L_0,L_1\}\subseteq \mathcal P(W)$, then the uniqueness of line intersection gives $W=L_0\cap L_1=D$. Hence any degenerate
rectangle containing all four edges would force $W=D\in n'=n$, contradicting $D\notin n$.
\end{proof}

\begin{proposition}[The strict $B_*$ witness isolates a skew determinant]
\label{prop:bstar-isolates-skew-delta}
Let $\Pi$ be a projective plane of order $q\ge 3$. Choose a point $D$, three distinct lines $L_0,L_1,L_2$ through $D$, a point
$B\in L_2\setminus\{D\}$, a point $C\in L_1\setminus\{D\}$, let $L_3:=\overline{BC}$, let $E:=L_0\cap L_3$, and choose
$A\in L_3\setminus\{B,C,E\}$. Then the $4\times4$ submatrix of $I(\Pi)$ with rows $(A,B,C,D)$ and columns $(L_0,L_1,L_2,L_3)$ is the
strict $(0,3)$ pattern $B_*$, with valuation-minimizing permutations
\[
\sigma_1=(0,1,2,3),\qquad \sigma_2=(1,0,2,3),\qquad \sigma_3=(2,1,0,3).
\]
For any rank-$\le 3$ valued-field lift $\widetilde I$ of $I(\Pi)$, the valuation-$0$ determinant initial form is equivalent to
\[
u_{C,L_2}\,\Delta_{AB}^{L_0L_1}=u_{A,L_2}u_{B,L_1}u_{C,L_0}
\qquad\text{in }k,
\tag{19.1}
\]
where
\[
\Delta_{AB}^{L_0L_1}:=u_{A,L_0}u_{B,L_1}-u_{A,L_1}u_{B,L_0}.
\tag{19.2}
\]
In particular, a strict $(0,3)$ witness of type $B_*$ isolates a \emph{skew} $2\times2$ determinant on the rectangle
$\{A,B\}\times\{L_0,L_1\}\subseteq S(L_3;D)$.
\end{proposition}

\begin{proof}
The existence of this $B_*$-configuration is exactly the construction already used in \cref{prop:bstar-abundance}. The three
valuation-minimal determinant monomials are
\[
M_1=u_{A,L_0}u_{B,L_1}u_{C,L_2}u_{D,L_3},
\quad
M_2=u_{A,L_1}u_{B,L_0}u_{C,L_2}u_{D,L_3},
\quad
M_3=u_{A,L_2}u_{B,L_1}u_{C,L_0}u_{D,L_3},
\]
with signs $+,-,-$, respectively. Since the lifted determinant vanishes, its initial form gives
\[
M_1-M_2-M_3=0.
\]
Factor the first two terms and cancel the nonzero common factor $u_{D,L_3}$:
\[
u_{C,L_2}u_{D,L_3}(u_{A,L_0}u_{B,L_1}-u_{A,L_1}u_{B,L_0})
=u_{A,L_2}u_{B,L_1}u_{C,L_0}u_{D,L_3},
\]
which is exactly \textup{(19.1)}. The points $A,B$ lie on $L_3$, the lines $L_0,L_1$ meet at $D$, and by construction $D\notin L_3$,
so the supporting $2\times2$ rectangle is skew rather than degenerate.
\end{proof}

\begin{remark}[Skew determinants and degenerate charts]
Equation \textup{(19.1)} has the form
\[
(\text{monomial})\cdot(\text{skew binomial})=(\text{monomial}).
\]
Degenerate square-minimizer gadgets control anchor rectangles of the form $R(n,W)$, whereas strict $B_*$ witnesses force control over a
skew determinant $\Delta_{AB}^{L_0L_1}$ on $S(L_3;D)$. Any global argument must therefore relate these two geometries and show how their interaction produces either positive cancellation depth or a nontrivial atlas holonomy defect.
\end{remark}

\begin{proposition}[A $\Lam$-inactive atlas covering a strict $(0,3)$ witness must fail]
\label{prop:atlas-forces-lam-activity}
Assume $\operatorname{char}(k)\notin\{2,3\}$. Let $M$ be a strict $(0,3)$ $4\times4$ minor of $I(\Pi)$, for example a $B_*$ witness, and let $E_0(M)$ be its valuation-$0$ support in $G_0(I(\Pi))$. Suppose there exists a $\Lam$-inactive rectangle atlas $\mathcal A$ such that:
\begin{enumerate}[leftmargin=2em,label=(\roman*)]
  \item $E_0(M)\subseteq \bigcup_{R\in\mathcal A} E(R)$;
  \item the chart adjacency graph of the subfamily used for this coverage is connected; and
  \item every chart cycle $\gamma$ has trivial holonomy $\operatorname{Hol}_{\mathcal A}(\gamma)=1$.
\end{enumerate}
Then the residues on $E_0(M)$ glue to one global factorization
\[u_{p,\ell}=\alpha_p\beta_\ell
\qquad\text{for all }(p,\ell)\in E_0(M),\]
so the valuation-$0$ initial form of $\det(\widetilde M)$ cannot vanish. Consequently, any such covering atlas must contain at least one chart that is \emph{$\Lam$-active}, i.e. not residue-pure-gauge.
\end{proposition}

\begin{proof}
By \cref{prop:atlas-trivial-holonomy}, hypotheses \textup{(i)}--\textup{(iii)} imply that the local chart factorizations glue to one global gauge on the union of the covering charts, hence in particular on $E_0(M)$. Therefore the valuation-$0$ residue labels on the strict $(0,3)$ support of $M$ satisfy the pure-gauge relation required by \cref{prop:pure-gauge-kills-strict03}. That proposition shows that, in characteristic $\neq2,3$, the initial form of $\det(\widetilde M)$ is then a nonzero sign sum and cannot vanish. This contradicts the existence of a rank-$\le3$ lift whose $4\times4$ minors all vanish. Hence at least one chart in every covering atlas must fail to be $\Lam$-inactive.
\end{proof}

\begin{remark}[Holonomy as an obstruction to global gluing]
\label{rem:atlas-holonomy-curvature}
The transition scalars from \cref{lem:atlas-transition-scalar} make local rank-$1$ rectangles behave like an atlas with multiplicative cocycle data. If the cocycle is holonomy-trivial, \cref{prop:atlas-forces-lam-activity} shows that some chart must become $\Lam$-active before a strict $(0,3)$ witness can be covered. If the cocycle is not holonomy-trivial, the obstruction is already visible as nontrivial chart holonomy. This is the natural holonomy obstruction in the rectangle-atlas formalism.
\end{remark}

\begin{proposition}[Holonomy detects the skew-binomial $\Lam$-event]
\label{prop:holonomy-detects-skew-delta}
Under the hypotheses of \cref{lem:bridge-cycle-holonomy-rho}, set
\[
\Delta_{AB}^{L_0L_1}:=u_{A,L_0}u_{B,L_1}-u_{A,L_1}u_{B,L_0}.
\]
Then
\[
\Delta_{AB}^{L_0L_1}
=u_{A,L_0}u_{B,L_1}\bigl(1-\operatorname{Hol}_{\mathcal A}(R_0R_1R_2R_3R_0)\bigr).
\tag{23.4}
\]
In particular,
\[
\operatorname{Hol}_{\mathcal A}(R_0R_1R_2R_3R_0)=1
\qquad\Longleftrightarrow\qquad
\Delta_{AB}^{L_0L_1}=0.
\tag{23.5}
\]
So the bridge-cycle holonomy equals $1$ exactly when the skew binomial cancels at residue level, i.e. exactly when this part of the lift pays a $\Lam$-depth event.
\end{proposition}

\begin{proof}
By \cref{lem:bridge-cycle-holonomy-rho},
\[
\operatorname{Hol}_{\mathcal A}(R_0R_1R_2R_3R_0)
=\Bigl(\frac{u_{A,L_1}}{u_{A,L_0}}\Bigr)
 \Bigl(\frac{u_{B,L_0}}{u_{B,L_1}}\Bigr).
\]
Multiplying by $u_{A,L_0}u_{B,L_1}$ gives
\[
u_{A,L_0}u_{B,L_1}\operatorname{Hol}_{\mathcal A}(R_0R_1R_2R_3R_0)=u_{A,L_1}u_{B,L_0},
\]
so
\[
u_{A,L_0}u_{B,L_1}\bigl(1-\operatorname{Hol}_{\mathcal A}(R_0R_1R_2R_3R_0)\bigr)
=u_{A,L_0}u_{B,L_1}-u_{A,L_1}u_{B,L_0}
=\Delta_{AB}^{L_0L_1}.
\]
Because $u_{A,L_0},u_{B,L_1}\in k^*$, the second displayed equivalence follows immediately.
\end{proof}

\begin{proposition}[The strict $B_*$ witness rewrites bridge holonomy affinely]
\label{prop:bstar-rewrites-bridge-holonomy}
Under the hypotheses of \cref{lem:bridge-cycle-holonomy-rho}, combine \cref{prop:bstar-isolates-skew-delta} with \cref{prop:holonomy-detects-skew-delta}. Then
\[
\operatorname{Hol}_{\mathcal A}(R_0R_1R_2R_3R_0)
=1-\Bigl(\frac{u_{C,L_0}}{u_{C,L_2}}\Bigr)
\Bigl(\frac{u_{A,L_2}}{u_{A,L_0}}\Bigr).
\tag{23.7}
\]
Equivalently, the observable holonomy of the bridge cycle is an affine function of the residue ratio $u_{A,L_2}/u_{A,L_0}$, with coefficient depending only on the fixed witness data $(C,L_0,L_2)$.
\end{proposition}

\begin{proof}
From \cref{prop:bstar-isolates-skew-delta} one has
\[
u_{C,L_2}\,\Delta_{AB}^{L_0L_1}=u_{A,L_2}u_{B,L_1}u_{C,L_0}.
\]
Substitute the identity from \cref{prop:holonomy-detects-skew-delta}:
\[
u_{C,L_2}\,u_{A,L_0}u_{B,L_1}\bigl(1-\operatorname{Hol}_{\mathcal A}(R_0R_1R_2R_3R_0)\bigr)
=u_{A,L_2}u_{B,L_1}u_{C,L_0}.
\]
Since $u_{B,L_1}\neq 0$, cancel it to obtain
\[
u_{C,L_2}u_{A,L_0}\bigl(1-\operatorname{Hol}_{\mathcal A}(R_0R_1R_2R_3R_0)\bigr)=u_{A,L_2}u_{C,L_0},
\]
and then solve for the holonomy.
\end{proof}

\begin{corollary}[Cross-ratio obstruction for a strict $B_*$ witness]
\label{cor:bstar-cross-ratio-obstruction}
Under the hypotheses of \cref{lem:bridge-cycle-holonomy-rho}, define
\[
\rho_{AB}^{L_0L_1}:=\frac{u_{A,L_1}}{u_{A,L_0}}\cdot\frac{u_{B,L_0}}{u_{B,L_1}}.
\]
Then
\[
\Delta_{AB}^{L_0L_1}=u_{A,L_0}u_{B,L_1}\bigl(1-\rho_{AB}^{L_0L_1}\bigr).
\tag{24.3}
\]
In particular, whenever the strict $(0,3)$ witness forces $\Delta_{AB}^{L_0L_1}\neq0$, one has
\[
\rho_{AB}^{L_0L_1}\neq1,
\]
so the zero rectangle $\{A,B\}\times\{L_0,L_1\}$ cannot be $\Lam$-inactive in residues.
\end{corollary}

\begin{proof}
By \cref{lem:bridge-cycle-holonomy-rho}, the displayed cross-ratio is exactly the bridge-cycle holonomy. Therefore \textup{(24.3)} is just \cref{prop:holonomy-detects-skew-delta} rewritten with $\rho_{AB}^{L_0L_1}$ in place of $\operatorname{Hol}_{\mathcal A}$. If $\Delta_{AB}^{L_0L_1}\neq0$, then \textup{(24.3)} implies $\rho_{AB}^{L_0L_1}\neq1$ because $u_{A,L_0}u_{B,L_1}\in k^*$. The final claim follows from \cref{lem:lam-inactive-implies-rho1}.
\end{proof}

\begin{lemma}[Each strict $B_*$ witness satisfies $\rho+\sigma=1$]
\label{lem:bstar-rho-plus-sigma}
Assume $\widetilde I$ is a rank-$\le 3$ valued-field lift of a projective-plane incidence matrix, and let
$(A,B,C,D;L_0,L_1,L_2,L_3)$ be a strict $B_*$ witness. Define
\[
\rho:=\rho_{AB}^{L_0L_1}=\frac{u_{A,L_1}}{u_{A,L_0}}\cdot\frac{u_{B,L_0}}{u_{B,L_1}},
\qquad
\sigma:=\sigma_{AC}^{L_0L_2}=\frac{u_{A,L_2}}{u_{A,L_0}}\cdot\frac{u_{C,L_0}}{u_{C,L_2}}.
\]
Then
\[
\rho+\sigma=1
\qquad\text{in }k.
\tag{26.1}
\]
In particular, $\rho\neq 1$ because $\sigma\in k^*$.
\end{lemma}

\begin{proof}
By \cref{prop:bstar-isolates-skew-delta},
\[
u_{C,L_2}\bigl(u_{A,L_0}u_{B,L_1}-u_{A,L_1}u_{B,L_0}\bigr)=u_{A,L_2}u_{B,L_1}u_{C,L_0}.
\]
Dividing by $u_{A,L_0}u_{B,L_1}u_{C,L_2}\in k^*$ gives
\[
1-\frac{u_{A,L_1}}{u_{A,L_0}}\frac{u_{B,L_0}}{u_{B,L_1}}
=
\frac{u_{A,L_2}}{u_{A,L_0}}\frac{u_{C,L_0}}{u_{C,L_2}},
\]
which is exactly $1-\rho=\sigma$.
\end{proof}

\begin{remark}[The affine relation $\rho+\sigma=1$]
\label{rem:rho-plus-sigma-hyperaddition}
Both $\rho$ and $\sigma$ are multiplicative gauge invariants built from valuation-$0$ residues, but \cref{lem:bstar-rho-plus-sigma} ties them together by one additive equation in $k$. This gives a concrete instance in which gauge-invariant multiplicative residue coordinates satisfy a nontrivial affine relation.
\end{remark}

\begin{remark}[Counting pivot]
\label{rem:counting-pivot-rho-ne-1}
The corollary packages a strict $B_*$ witness into one local obstruction statement: every witness with nonzero skew determinant produces one $2\times2$ zero rectangle whose cross-ratio is forced away from $1$. The next step is to count how many distinct rectangles with $\rho\neq1$ can be forced, and how many such rectangles one $\Lam$-active event can explain.
\end{remark}

\begin{definition}[$B_*$ witness data]
\label{def:bstar-witness-data}
Fix a point $D$ and an ordered triple of distinct lines $(L_0,L_1,L_2)$ through $D$. Choose
\begin{itemize}[leftmargin=2em]
  \item $B\in L_2\setminus\{D\}$,
  \item $C\in L_1\setminus\{D\}$,
\end{itemize}
set $L_3:=\overline{BC}$ and $E:=L_0\cap L_3$, and then choose
\begin{itemize}[leftmargin=2em]
  \item $A\in L_3\setminus\{B,C,E\}$.
\end{itemize}
The resulting ordered datum $(A,B,C,D;L_0,L_1,L_2,L_3)$ is called a \emph{$B_*$ witness datum}. Its associated $4\times4$ incidence submatrix is exactly the strict $(0,3)$ pattern already isolated in the baseline $B_*$ analysis.
\end{definition}

\begin{lemma}[$B_*$ witnesses are exactly skew $2\times2$ rectangles]
\label{lem:bstar-bijection-skew-rectangle}
Fix an ordered quadruple $(A,B;L_0,L_1)$ such that all four entries in
$\{A,B\}\times\{L_0,L_1\}$ are tropical $0$, and let
\[
D:=L_0\cap L_1.
\]
Assume additionally that
\[
D\notin \overline{AB}
\qquad\text{equivalently}\qquad
A\notin \overline{BD}.
\]
Then there is a unique $B_*$ witness extending this data, namely
\[
L_2:=\overline{BD},
\qquad
L_3:=\overline{AB},
\qquad
C:=L_1\cap L_3,
\qquad
E:=L_0\cap L_3.
\]
Conversely, every $B_*$ witness datum determines such an ordered skew rectangle $(A,B;L_0,L_1)$.
\end{lemma}

\begin{proof}
The point $D=L_0\cap L_1$ is forced by the projective-plane axioms. Then $L_2=\overline{BD}$ and
$L_3=\overline{AB}$ are uniquely determined. The points $C=L_1\cap L_3$ and $E=L_0\cap L_3$ are therefore also forced.
Because $D\notin L_3$, one has $L_2\neq L_3$, and the incidences/nonincidences are exactly those of the strict $B_*$ pattern:
$D\in L_0,L_1,L_2$ but $D\notin L_3$; $B\in L_2,L_3$ but $B\notin L_0,L_1$; $C\in L_1,L_3$ but $C\notin L_0,L_2$;
and $A\in L_3$ but $A\notin L_0,L_1,L_2$. So the completion is unique and produces a $B_*$ witness. The converse is immediate from the construction of $B_*$ witness data in \cref{def:bstar-witness-data}.
\end{proof}

\begin{remark}[Witness count equals skew-rectangle count]
\label{rem:bstar-count-skew-rectangles}
By \cref{lem:bstar-bijection-skew-rectangle,cor:bstar-total-count}, the exact $\Theta(q^8)$ count of ordered $B_*$ witnesses is equally a count of ordered skew all-zero $2\times2$ rectangles $(A,B;L_0,L_1)$ satisfying $D=L_0\cap L_1\notin\overline{AB}$.
\end{remark}

\begin{lemma}[Exact count of $B_*$ witnesses for fixed $(D;L_0,L_1,L_2)$]
\label{lem:bstar-count-fixed-triple}
Let $\Pi_q$ be a projective plane of order $q\ge 3$. Fix a point $D$ and an ordered triple of distinct lines $(L_0,L_1,L_2)$ through $D$. Then the number of choices $(B,C,A)$ producing a $B_*$ witness datum is
\[
q^2(q-2).
\]
\end{lemma}

\begin{proof}
There are $q$ choices for $B\in L_2\setminus\{D\}$ and $q$ choices for $C\in L_1\setminus\{D\}$. Once $B$ and $C$ are chosen, the line $L_3=\overline{BC}$ contains $q+1$ points. The points $B,C,E$ are pairwise distinct: $B\notin L_0$, $C\notin L_0$, and $L_0\cap L_1=L_0\cap L_2=D$. Hence there are $(q+1)-3=q-2$ valid choices of
\[
A\in L_3\setminus\{B,C,E\}.
\]
Multiplying gives $q\cdot q\cdot(q-2)=q^2(q-2)$.
\end{proof}

\begin{corollary}[Total number of $B_*$ witnesses is $\Theta(q^8)$]
\label{cor:bstar-total-count}
Let $\Pi_q$ be a projective plane of order $q\ge 3$ and write $v:=q^2+q+1$. Then the total number of ordered $B_*$ witness data is
\[
N_{B_*}(q)=v(q+1)q(q-1)\cdot q^2(q-2)
=(q^2+q+1)(q+1)q(q-1)q^2(q-2)=\Theta(q^8).
\]
\end{corollary}

\begin{proof}
There are $v=q^2+q+1$ choices for the distinguished point $D$. Through each $D$ pass $q+1$ lines, so the number of ordered distinct triples $(L_0,L_1,L_2)$ through $D$ is
\[
(q+1)q(q-1).
\]
Now apply \cref{lem:bstar-count-fixed-triple}.
\end{proof}

\begin{lemma}[Each $B_*$ witness forces $\rho\neq 1$ in characteristic $\neq2$]
\label{lem:bstar-witness-forces-rho-ne-1}
Assume $\operatorname{char}(k)\neq 2$. Let $(A,B,C,D;L_0,L_1,L_2,L_3)$ be any $B_*$ witness datum arising from a rank-$\le 3$ lift of the incidence matrix. Then the $2\times 2$ all-zero rectangle $\{A,B\}\times\{L_0,L_1\}$ satisfies
\[
\rho_{AB}^{L_0L_1}\neq 1.
\]
\end{lemma}

\begin{proof}
The datum is, by definition, a strict $B_*$ witness. So \cref{prop:bstar-isolates-skew-delta} gives
\[
u_{C,L_2}\,\Delta_{AB}^{L_0L_1}=u_{A,L_2}u_{B,L_1}u_{C,L_0},
\]
where all three factors on the right are nonzero residues of tropical-$0$ entries. Hence $\Delta_{AB}^{L_0L_1}\neq 0$. Now apply \cref{cor:bstar-cross-ratio-obstruction}.
\end{proof}

\begin{remark}[Asymptotic supply of local obstructions]
\label{rem:bstar-theta-q8-obstruction-supply}
Taken together, \cref{cor:bstar-total-count,lem:bstar-witness-forces-rho-ne-1} show that, in characteristic $\neq2$, every rank-$\le3$ lift of the incidence matrix of a projective plane of order $q\ge 3$ must accommodate
\[
\Theta(q^8)
\]
ordered $B_*$ witness data, and every one of them forces one gauge-invariant local rectangle constraint of the form $\rho\neq 1$. This yields the first polynomial-scale family of gauge-invariant local residue obstructions established here.
\end{remark}

\begin{remark}[Closed algebraic loop: depth versus holonomy]
\label{rem:holonomy-delta-closed-loop}
The bridge cycle provides the first closed algebraic loop in the paper:
\[
\text{observable chart holonomy}
\Longleftrightarrow
\text{skew determinant }\Delta_{AB}^{L_0L_1}
\Longleftrightarrow
\text{strict }B_*\text{ witness equation}.
\]
If the bridge-cycle holonomy is $1$, then the skew determinant cancels and the lift pays a $\Lam$-depth event. If the holonomy is not $1$, then the same local witness detects genuine atlas curvature. This dichotomy makes the bridge cycle a particularly transparent local interface between depth and curvature in the present formalism.
\end{remark}

\begin{remark}[Modular holonomy and bracket sensitivity]
\label{rem:modular-holonomy-associator}
The M\"obius element $g_\gamma\in PGL_2(\mathbb Z)$ attached to a witness loop is the first bracket-sensitive invariant isolated in the residue formalism: different compositions of the same local moves can force nontrivial fixed-point equations.
\end{remark}

\begin{lemma}[Scaled modular shortcut gives a quadratic constraint]
\label{lem:scaled-modular-shortcut-quadratic}
Suppose two rectangle states $R,R'$ with cross-ratios $x:=x_R$ and $x':=x_{R'}$ are related in two independent ways:
\begin{enumerate}[leftmargin=2em]
  \item a strict-$B_*$ plus swap path giving
  \[
  x'=SI(x)=1-\frac{1}{x};
  \]
  \item a base-line transport path giving
  \[
  x'=M_\lambda(x)=\lambda x
  \qquad\text{for some }\lambda\in k^*.
  \]
\end{enumerate}
Then consistency forces
\[
1-\frac1x=\lambda x
\qquad\Longleftrightarrow\qquad
\lambda x^2-x+1=0.
\tag{29.2}
\]
\end{lemma}

\begin{proof}
Equate the two expressions for $x'$ and clear denominators:
\[
1-\frac1x=\lambda x
\iff x-1=\lambda x^2
\iff \lambda x^2-x+1=0.
\]
\end{proof}

\begin{remark}[Field-dependent quadratic obstruction template]
\label{rem:quadratic-shortcut-template}
The equation
\[
\lambda x^2-x+1=0
\]
is the first residue-level obstruction in the baseline whose strength depends explicitly on the residue field. For example, over $k=\mathbb R$, positive scalings with $\lambda>1/4$ admit no real solution. Thus a geometrically realized shortcut of that type would force a $\Lam$-event rather than a residue-level solution. The remaining bottleneck is geometric: manufacture the transport edge $M_\lambda$ in a gauge-invariant way and at scale.
\end{remark}

\begin{definition}[Rectangle $\Lam$-depth scalar]
\label{def:rectangle-lam-depth-scalar}
For an all-zero $2\times2$ rectangle $\{A,B\}\times\{L_0,L_0'\}$, define the lifted determinant
\[
\Delta:=\widetilde u_{A,L_0}\widetilde u_{B,L_0'}-\widetilde u_{A,L_0'}\widetilde u_{B,L_0}.
\]
Its associated transport-depth scalar is
\[
\delta_{AB}^{L_0L_0'}:=\nu(\Delta)-\min\bigl(\nu(\widetilde u_{A,L_0}\widetilde u_{B,L_0'}),\nu(\widetilde u_{A,L_0'}\widetilde u_{B,L_0})\bigr)\ge 0.
\tag{30.5}
\]
When the valuation-minimal terms both have valuation $0$, this simplifies to $\delta_{AB}^{L_0L_0'}=\nu(\Delta)$. It is the lift-side analogue of stored cancellation mass: $\delta>0$ exactly when the transport rectangle is residue-flat and extra cancellation has been absorbed into the $\Lam$-layer.
\end{definition}

\begin{definition}[Tie depth for an equal-valuation binomial]
\label{def:tie-depth-binomial}
Let $(K,\nu)$ be a nonarchimedean valued field. For $x,y\in K^\times$ with
\[
\nu(x)=\nu(y)=\alpha,
\]
define the \emph{tie depth} of $x+y$ by
\[
\delta(x+y):=\nu(x+y)-\alpha\in \RR_{\ge 0}\cup\{\infty\}.
\]
Thus $\delta(x+y)>0$ records a valuation jump created by leading-term cancellation.
\end{definition}

\begin{lemma}[A $2\times2$ binomial is controlled by one cross-ratio and one tie depth]
\label{lem:2x2-crossratio-tiedepth}
Let $a,b,c,d\in K^\times$ satisfy
\[
\nu(a)=\nu(b)=\nu(c)=\nu(d)=0,
\qquad
\rho:=\frac{ad}{bc}\in K^\times.
\]
Then:
\begin{enumerate}[leftmargin=2em]
  \item \[\nu(ad-bc)=\nu(\rho-1).\]
  \item \[\nu(ad-bc)>0 \iff \overline\rho=1\text{ in }k^*,\]
  equivalently $\operatorname{lc}(a)\operatorname{lc}(d)=\operatorname{lc}(b)\operatorname{lc}(c)$ in the residue field.
  \item The tie depth of the binomial $ad-bc$ is exactly
  \[\delta(ad-bc)=\nu(\rho-1).\]
\end{enumerate}
\end{lemma}

\begin{proof}
Since $\nu(bc)=0$, we can factor
\[
ad-bc=bc\Bigl(\frac{ad}{bc}-1\Bigr)=bc(\rho-1).
\]
Taking valuations gives
\[
\nu(ad-bc)=\nu(bc)+\nu(\rho-1)=\nu(\rho-1).
\]
Because $\rho\in R_K^\times$, the inequality $\nu(\rho-1)>0$ holds if and only if $\rho\equiv 1\pmod{\mathfrak m_K}$, i.e. if and only if $\overline\rho=1$ in $k^*$. The final statement is just the definition of tie depth from \cref{def:tie-depth-binomial} with $x=ad$ and $y=-bc$.
\end{proof}

\begin{remark}[A direct bridge to cancellation depth]
\label{rem:note31-cancellation-bridge}
The equality-valued binomial $ad-bc$ gives the simplest lift-theoretic model of two leading terms that tie and may or may not cancel. The residue condition $\overline\rho=1$ detects whether cancellation occurs at leading order, while the scalar $\nu(\rho-1)$ measures the resulting depth of that cancellation.
\end{remark}


\subsection{Gauge normalization and flat charts}
\begin{definition}[Gauge transformations]
A gauge transformation is multiplication of $F$ on the left/right by diagonal matrices of valuation $0$
(i.e.\ unit rescalings of rows/columns), preserving the valuation pattern.
\end{definition}

\begin{definition}[Flat rectangle / cross-ratio]
Given a $2\times2$ all-zero rectangle in the valuation pattern, define its residue cross-ratio
$\rho:=\frac{u_{11}u_{22}}{u_{12}u_{21}}\in k^\times$. We call the rectangle \emph{flat} if $\rho=1$.
\end{definition}

\begin{figure}[htbp]
\centering
\begin{tikzpicture}[font=\small]
  \matrix (m) [matrix of math nodes,
               row sep=1.8em,
               column sep=2.6em,
               left delimiter={(},
               right delimiter={)}] {
      u_{11} & u_{12} \\
      u_{21} & u_{22} \\
  };
  \draw[dashed] ($(m-1-1.north west)+(-3pt,3pt)$)
      rectangle
      ($(m-2-2.south east)+(3pt,-3pt)$);

  \node (rho) [below=1.4em of m] {$\rho=\dfrac{u_{11}u_{22}}{u_{12}u_{21}}$};
\end{tikzpicture}
\caption{A valuation-minimal $2\times2$ zero rectangle and its associated residue cross-ratio. 
The flatness condition is $u_{11}u_{22}=u_{12}u_{21}$, equivalently $\rho=1$.}
\label{fig:flat-rectangle}
\end{figure}



\subsection{Zero-graph holonomy and pure-gauge factorization}
\begin{definition}[Zero-graph / nonincidence graph]
Let $B\in\{0,1\}^{R\times C}$ be any $0$--$1$ matrix, viewed so that the entries equal to $0$ are the valuation-minimal ones.
Its \emph{zero-graph} $G_0(B)$ is the bipartite graph with row-vertex set $R$, column-vertex set $C$, and edge $(r,c)$ whenever
$B_{r,c}=0$. For a lift $\widetilde B$ with $\val(\widetilde B_{r,c})=B_{r,c}$, we write
\[
u_{r,c}:=\operatorname{lc}(\widetilde B_{r,c})\in k^*
\]
for every zero-edge $(r,c)\in E(G_0(B))$. In particular, for a projective plane we apply this to the combinatorial incidence matrix
$I(\Pi)$, so $G_0(I(\Pi))$ is exactly the nonincidence graph on points and lines.
\end{definition}

\begin{definition}[Alternating holonomy on an even cycle]
Let
\[
C: p_1-\ell_1-p_2-\ell_2-\cdots-p_m-\ell_m-p_1
\]
be an even cycle in a bipartite zero-graph, with indices read modulo $m$. Its alternating holonomy is
\[
\operatorname{Hol}_{u}(C):=\prod_{i=1}^m \frac{u_{p_i,\ell_i}}{u_{p_{i+1},\ell_i}}\in k^*.
\]
Equivalently, $\operatorname{Hol}_{u}(C)=1$ says that the product of every other edge label around $C$ equals the product of the
remaining alternating edge labels.
\end{definition}

\begin{definition}[Signed holonomy on an even cycle]
\label{def:signed-holonomy}
For an even cycle
\[
C: p_1-\ell_1-p_2-\ell_2-\cdots-p_m-\ell_m-p_1
\]
in a bipartite zero-graph, define its \emph{signed holonomy} by
\[
\operatorname{sHol}_{u}(C):=(-1)^m\,\operatorname{Hol}_{u}(C)\in k^*.
\]
Thus $\operatorname{sHol}_{u}(C)=1$ is equivalent to $\operatorname{Hol}_{u}(C)=(-1)^m$.
For the $4\times4$ two-minimum minors studied below, this is exactly the determinant-sign-corrected form of the first-layer
cancellation constraint.
\end{definition}

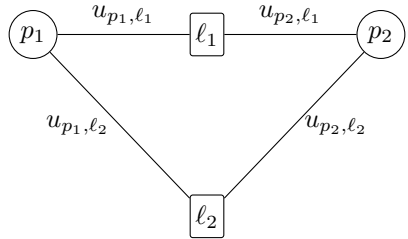
\begin{figure}[htbp]
\centering
\begin{tikzpicture}[font=\small,
  point/.style={circle,draw,inner sep=1.5pt,minimum size=18pt},
  linev/.style={rectangle,draw,rounded corners=1pt,inner sep=2pt,minimum height=16pt},
  every edge/.style={draw,-}]
  \node[point] (p1) at (0,1.2) {$p_1$};
  \node[linev] (l1) at (2.3,1.2) {$\ell_1$};
  \node[point] (p2) at (4.6,1.2) {$p_2$};
  \node[linev] (l2) at (2.3,-1.2) {$\ell_2$};
  \draw (p1) -- node[above] {$u_{p_1,\ell_1}$} (l1);
  \draw (l1) -- node[above] {$u_{p_2,\ell_1}$} (p2);
  \draw (p2) -- node[right] {$u_{p_2,\ell_2}$} (l2);
  \draw (l2) -- node[left] {$u_{p_1,\ell_2}$} (p1);
  \node[below=1.4em of l2, align=center] {$\operatorname{Hol}_u(C)=\dfrac{u_{p_1,\ell_1}}{u_{p_2,\ell_1}}\dfrac{u_{p_2,\ell_2}}{u_{p_1,\ell_2}}$};
\end{tikzpicture}
\caption{Alternating holonomy on the smallest even cycle of a bipartite zero-graph.}
\label{fig:holonomy-cycle}
\end{figure}

\begin{lemma}[Connectivity and diameter of the nonincidence graph]
\label{lem:g0-connected}
Let $\Pi_q$ be a projective plane of order $q\ge 2$. Then $G_0(I(\Pi_q))$ is connected. More precisely:
\begin{enumerate}[leftmargin=2em]
  \item any two distinct point-vertices are at distance $2$;
  \item any two distinct line-vertices are at distance $2$;
  \item any incident point-line pair is at distance $3$.
\end{enumerate}
In particular,
\[
\operatorname{diam}(G_0(I(\Pi_q)))=3.
\]
\end{lemma}

\begin{proof}
Let $p\neq p'$ be distinct points. Since exactly $q+1$ lines pass through each point and exactly one line passes through both,
the number of lines through $p$ or $p'$ is
\[
(q+1)+(q+1)-1=2q+1.
\]
As the total number of lines is $q^2+q+1$, the number of lines through neither point is
\[
(q^2+q+1)-(2q+1)=q(q-1)>0.
\]
Choosing such a line $\ell$ gives a path $p-\ell-p'$ in $G_0(I(\Pi_q))$, so every two distinct points are at distance $2$.
By point-line duality, the same holds for any two distinct lines.

Now let $p\in \ell$ be incident. Choose any line $m$ not through $p$; such lines exist because only $q+1$ of the $q^2+q+1$ lines
pass through $p$. The union $\ell\cup m$ contains $(q+1)+(q+1)-1=2q+1$ points, strictly fewer than the total
$q^2+q+1$, so there exists a point $p'$ on neither $\ell$ nor $m$. Then
\[
p-m-p'-\ell
\]
is a path of length $3$ in $G_0(I(\Pi_q))$. Because $G_0(I(\Pi_q))$ is bipartite and $p,\ell$ lie in opposite parts, their distance
must be odd; since they are incident, it is not $1$, hence it is exactly $3$. Therefore the graph is connected and has diameter $3$.
\end{proof}

\begin{lemma}[Holonomy-trivial $\Longleftrightarrow$ edge factorization on a connected bipartite graph]
\label{lem:holonomy-factorization}
Assume $G_0(B)$ is connected. Then the following are equivalent:
\begin{enumerate}[leftmargin=2em]
  \item For every even cycle $C\subseteq G_0(B)$, one has $\operatorname{Hol}_{u}(C)=1$.
  \item There exist functions $\alpha:R\to k^*$ and $\beta:C\to k^*$ such that
  \[
  u_{r,c}=\alpha_r\beta_c
  \]
  for every edge $(r,c)\in E(G_0(B))$.
\end{enumerate}
\end{lemma}

\begin{proof}
If $u_{r,c}=\alpha_r\beta_c$, then each cycle factor simplifies to
\[
\frac{u_{p_i,\ell_i}}{u_{p_{i+1},\ell_i}}=\frac{\alpha_{p_i}\beta_{\ell_i}}{\alpha_{p_{i+1}}\beta_{\ell_i}}=\frac{\alpha_{p_i}}{\alpha_{p_{i+1}}},
\]
so the product telescopes around the cycle and gives $\operatorname{Hol}_{u}(C)=1$.

Conversely, choose a spanning tree $T$ of $G_0(B)$ and a root row-vertex $r_\star$. Set $\alpha_{r_\star}=1$. Traverse the tree as
follows: whenever $(r,c)\in T$ and $\alpha_r$ is known, define $\beta_c:=u_{r,c}/\alpha_r$; whenever $(r,c)\in T$ and $\beta_c$ is
known, define $\alpha_r:=u_{r,c}/\beta_c$. Because $T$ is connected and acyclic, this assigns $\alpha$ and $\beta$ uniquely on all
vertices of $T$.

Now take any off-tree edge $e=(r,c)\notin T$. Adding $e$ to $T$ creates a unique cycle $C_e$. The condition
$\operatorname{Hol}_{u}(C_e)=1$ is exactly the consistency statement that the tree-defined product along the path from $r$ to $c$
reconstructs the given label on $e$, i.e. that $u_{r,c}=\alpha_r\beta_c$. Hence the same factorization holds on every edge of
$G_0(B)$.
\end{proof}

\begin{remark}[Disconnected case]
If $G_0(B)$ has several connected components, the same argument applies on each component separately. Thus connectedness is only used to
obtain one global pair of potentials rather than a componentwise family.
\end{remark}

\begin{definition}[Square-connected subgraph]
\label{def:square-connected}
Let $H\subseteq G_0(B)$ be a connected bipartite subgraph. We call $H$ \emph{square-connected} if its cycle space is generated by the
$4$-cycles contained in $H$. Equivalently, every cycle-holonomy constraint on $H$ is already forced by the square relations coming from
those $4$-cycles.
\end{definition}

\begin{corollary}[Square relations force pure gauge on a square-connected region]
\label{cor:square-relations-pure-gauge}
Let $H\subseteq G_0(B)$ be square-connected, and let $u_e\in k^*$ label each edge $e\in E(H)$. Assume that every $4$-cycle
\[
C: p_1-\ell_1-p_2-\ell_2-p_1
\]
in $H$ satisfies
\[
u_{p_1,\ell_1}u_{p_2,\ell_2}=u_{p_1,\ell_2}u_{p_2,\ell_1}.
\tag{18.1}
\]
Then there exist scalars $\alpha_p,\beta_\ell\in k^*$ such that
\[
u_{p,\ell}=\alpha_p\beta_\ell
\qquad\text{for every edge }(p,\ell)\in E(H).
\tag{18.2}
\]
In other words, rank-$1$ rectangles glue to a pure-gauge chart on every square-connected covered region.
\end{corollary}

\begin{proof}
Each square relation \textup{(18.1)} is exactly the condition $\operatorname{Hol}_{u}(C)=1$ on the corresponding $4$-cycle. Since $H$ is
square-connected, those $4$-cycles generate the cycle space of $H$, so every cycle in $H$ has trivial unsigned holonomy. Applying
\cref{lem:holonomy-factorization} on $H$ gives the desired factorization \textup{(18.2)}.
\end{proof}

\begin{remark}[Path independence versus local square gluing]
\label{rem:path-independence-square-gluing}
\Cref{cor:square-relations-pure-gauge} is the graph-theoretic propagation principle behind the degenerate rank-$1$ rectangles from
\cref{cor:rank-one-rectangle}. Once enough such rectangles overlap to generate a square-connected region, the local $2\times2$ vanishing
conditions stop being isolated and become one cocycle-triviality statement $u_{p,\ell}=\alpha_p\beta_\ell$ on the whole covered region.
\end{remark}

\begin{lemma}[Structure of $(0,2)$-type $4\times4$ zero patterns]
\label{lem:02-single-cycle}
Let $B\in\{0,1\}^{4\times4}$ be of type $(0,2)$, and let $\sigma,\tau\in S_4$ be its two weight-$0$ permutations. Then:
\begin{enumerate}[leftmargin=2em]
  \item the union of the two zero-matchings is a single even cycle of length $2k$ with $k\in\{2,3,4\}$;
  \item the permutation-sign ratio satisfies
  \[
  \frac{\operatorname{sgn}(\tau)}{\operatorname{sgn}(\sigma)}=(-1)^{k-1}.
  \]
\end{enumerate}
\end{lemma}

\begin{proof}
The symmetric difference of two perfect matchings in a bipartite graph is always a disjoint union of even cycles. If there were at least
two such cycles, one could flip the two matchings independently on each component and obtain at least four distinct perfect matchings in
the zero pattern, contradicting the assumption that $B$ has exactly two weight-$0$ permutations. Hence the symmetric difference is a
single even cycle. In a $4\times4$ bipartite graph, its length can only be $4$, $6$, or $8$, so it is of the form $2k$ with
$k\in\{2,3,4\}$.

On the support of that cycle, the permutation $\tau\sigma^{-1}$ is a single $k$-cycle (with fixed points if $k<4$). A $k$-cycle has sign
$(-1)^{k-1}$, so
\[
\operatorname{sgn}(\tau)\operatorname{sgn}(\sigma)^{-1}=\operatorname{sgn}(\tau\sigma^{-1})=(-1)^{k-1}.
\]
\end{proof}

\begin{proposition}[Two-minimum $4\times4$ minors impose signed holonomy]
\label{prop:02-signed-holonomy}
Let $B\in\{0,1\}^{4\times4}$ be of type $(0,2)$, and let $\widetilde B$ be a valued-field lift with $\val(\widetilde B)=B$.
Let $\sigma,\tau$ be the two weight-$0$ permutations, let
\[
U_{\sigma}:=\prod_{i=1}^4 \operatorname{lc}(\widetilde B_{i,\sigma(i)}),
\qquad
U_{\tau}:=\prod_{i=1}^4 \operatorname{lc}(\widetilde B_{i,\tau(i)}),
\]
and let $2k$ be the length of the unique cycle from \cref{lem:02-single-cycle}. Then
\[
\val(\det \widetilde B)>0
\qquad\Longleftrightarrow\qquad
\frac{U_{\sigma}}{U_{\tau}}=(-1)^k.
\]
Equivalently, after ordering that cycle as
\[
C:p_1-\ell_1-p_2-\ell_2-\cdots-p_k-\ell_k-p_1
\]
so that one minimizing matching uses the edges $(p_i,\ell_i)$ and the other uses $(p_{i+1},\ell_i)$, one has
\[
\operatorname{Hol}_{u}(C)=(-1)^k
\qquad\Longleftrightarrow\qquad
\operatorname{sHol}_{u}(C)=1.
\]
\end{proposition}

\begin{proof}
By \cref{lem:two-term-det},
\[
\val(\det \widetilde B)>0
\qquad\Longleftrightarrow\qquad
\operatorname{sgn}(\sigma)U_{\sigma}+\operatorname{sgn}(\tau)U_{\tau}=0.
\]
Dividing by $\operatorname{sgn}(\sigma)U_{\tau}$ yields
\[
\frac{U_{\sigma}}{U_{\tau}}=-\frac{\operatorname{sgn}(\tau)}{\operatorname{sgn}(\sigma)}.
\]
Now apply \cref{lem:02-single-cycle} to replace the sign ratio by $(-1)^{k-1}$ and obtain
\[
\frac{U_{\sigma}}{U_{\tau}}=(-1)^k.
\]
With the displayed ordering of the cycle, the left-hand side is exactly $\operatorname{Hol}_{u}(C)$, so the signed-holonomy reformulation
follows directly from \cref{def:signed-holonomy}.
\end{proof}

\begin{remark}[Characteristic-$2$ shadow]
If $\operatorname{char}(k)=2$, then $+1=-1$ in the residue field, so the signed-holonomy distinction collapses and every $(0,2)$-type
minor again imposes the unsigned condition $\operatorname{Hol}_{u}(C)=1$. Thus the pure-gauge picture survives verbatim in
characteristic $2$.
\end{remark}

\begin{definition}[$\Lam$-inactive degenerate chart]
\label{def:lam-inactive-chart}
Let $\Pi$ be a projective plane. A degenerate rectangle $R(n,W)$ from \cref{def:degenerate-vs-skew-rectangles} is called \emph{$\Lam$-inactive} if its residue labels admit a factorization
\[
u_{Z,r}=\alpha_Z\beta_r
\qquad\text{for every edge }(Z,r)\in R(n,W),\]
with $\alpha_Z,\beta_r\in k^*$. Equivalently, every $2\times2$ residue determinant on $R(n,W)$ vanishes, so the chart is already pure gauge at residue level.
\end{definition}

\begin{lemma}[Transition scalar on a connected overlap]
\label{lem:atlas-transition-scalar}
Let $H\subseteq G_0(I(\Pi))$ be a connected subgraph with at least one edge, and suppose that on $E(H)$ one has two factorizations
\[u_{p,\ell}=\alpha_p\beta_\ell=\alpha'_p\beta'_\ell
\qquad\text{for all }(p,\ell)\in E(H),\]
with all $\alpha,\beta,\alpha',\beta'\in k^*$. Then there exists a scalar $c\in k^*$ such that
\[
\alpha'_p=c\alpha_p \quad (p\in P\cap V(H)),
\qquad
\beta'_\ell=c^{-1}\beta_\ell \quad (\ell\in L\cap V(H)).
\]
\end{lemma}

\begin{proof}
Choose one edge $(p_0,\ell_0)\in E(H)$ and define
\[
c:=\frac{\alpha'_{p_0}}{\alpha_{p_0}}=\frac{\beta_{\ell_0}}{\beta'_{\ell_0}},
\]
where the equality follows from $\alpha_{p_0}\beta_{\ell_0}=\alpha'_{p_0}\beta'_{\ell_0}$. If $(p_0,\ell)$ is any other edge of $H$, then
\[
\alpha_{p_0}\beta_{\ell}=\alpha'_{p_0}\beta'_{\ell}=c\alpha_{p_0}\beta'_{\ell},
\]
so $\beta'_{\ell}=c^{-1}\beta_{\ell}$. Similarly, from any edge $(p,\ell)$ for which $\beta'_{\ell}=c^{-1}\beta_{\ell}$ is already known, one gets
\[
\alpha_p\beta_{\ell}=\alpha'_p\beta'_{\ell}=\alpha'_p c^{-1}\beta_{\ell},
\]
so $\alpha'_p=c\alpha_p$. Since $H$ is connected, this propagates across all vertices of $H$.
\end{proof}

\begin{definition}[Rectangle atlas and chart holonomy]
\label{def:rectangle-atlas-holonomy}
A \emph{$\Lam$-inactive rectangle atlas} on a subgraph of $G_0(I(\Pi))$ is a finite family of $\Lam$-inactive degenerate charts
\[
\mathcal A=\{R_i\}_{i\in I},
\]
each equipped with one factorization $u=\alpha^{(i)}\beta^{(i)}$ on $R_i$, such that every nonempty overlap $R_i\cap R_j$ used by the atlas is connected. The \emph{chart adjacency graph} has vertex set $I$ and an edge $i\sim j$ whenever $R_i\cap R_j\neq\varnothing$.

On each nonempty overlap, \cref{lem:atlas-transition-scalar} gives a unique transition scalar $c_{ij}\in k^*$ with
\[
\alpha^{(j)}=c_{ij}\alpha^{(i)},
\qquad
\beta^{(j)}=c_{ij}^{-1}\beta^{(i)}
\]
on $R_i\cap R_j$. For a cycle of charts
\[
\gamma: i_0\to i_1\to\cdots\to i_m=i_0
\]
in the adjacency graph, define its \emph{chart holonomy} by
\[
\operatorname{Hol}_{\mathcal A}(\gamma):=\prod_{t=0}^{m-1} c_{i_t i_{t+1}}\in k^*.
\]
If the local factorizations glue to one global gauge on the union of the charts, then necessarily $\operatorname{Hol}_{\mathcal A}(\gamma)=1$ for every chart cycle $\gamma$.
\end{definition}

\begin{proposition}[Trivial atlas holonomy glues local gauges]
\label{prop:atlas-trivial-holonomy}
Let $\mathcal A=\{R_i\}_{i\in I}$ be a $\Lam$-inactive rectangle atlas whose chart adjacency graph is connected. Assume that every chart cycle $\gamma$ satisfies
\[
\operatorname{Hol}_{\mathcal A}(\gamma)=1.
\]
Then there exist scalars $\alpha_p,\beta_\ell\in k^*$ such that
\[u_{p,\ell}=\alpha_p\beta_\ell
\qquad\text{for every edge }(p,\ell)\text{ lying in the union of the charts.}
\]
\end{proposition}

\begin{proof}
Choose one root chart $R_{i_0}$ and keep its factorization $\alpha^{(i_0)},\beta^{(i_0)}$. For any other chart $R_j$, pick a path
\[
i_0\to i_1\to\cdots\to i_t=j
\]
in the chart adjacency graph and define
\[
\lambda_j:=\prod_{s=0}^{t-1} c_{i_s i_{s+1}}.
\]
Because every chart cycle has holonomy $1$, this scalar is independent of the chosen path. Rescale the factorization on $R_j$ by setting
\[
\widehat\alpha^{(j)}:=\lambda_j^{-1}\alpha^{(j)},
\qquad
\widehat\beta^{(j)}:=\lambda_j\beta^{(j)}.
\]
On an overlap $R_i\cap R_j$, the transition rule from \cref{lem:atlas-transition-scalar} gives
\[
\widehat\alpha^{(j)}=\widehat\alpha^{(i)},
\qquad
\widehat\beta^{(j)}=\widehat\beta^{(i)},
\]
so the rescaled local gauges agree on every overlap. Hence they glue to one global pair $(\alpha_p,\beta_\ell)$ on the union of the charts, and the displayed factorization follows.
\end{proof}

\begin{definition}[Trimmed degenerate chart]
\label{def:trimmed-degenerate-chart}
Let $n$ be a line, let $W\in n$, and let $\ell\in \mathcal P(W)\setminus\{n\}$. Define the \emph{trimmed degenerate chart}
\[
R(n,W;\ell):=(n\setminus\{W\})\times\bigl(\mathcal P(W)\setminus\{n,\ell\}\bigr)\subseteq R(n,W).
\]
Whenever the ambient degenerate rectangle $R(n,W)$ is $\Lam$-inactive, every trimmed chart $R(n,W;\ell)$ inherits the same residue
factorization and is therefore $\Lam$-inactive as well.
\end{definition}

\begin{lemma}[A concrete three-chart overlap cycle around a strict $B_*$ witness]
\label{lem:three-chart-cycle-bstar}
Assume $q\ge 6$, and fix one strict $B_*$ witness with rows $(A,B,C,D)$ and columns $(L_0,L_1,L_2,L_3)$ as in
\cref{prop:bstar-isolates-skew-delta}. Then one can choose points
\[
W_0\in L_0\setminus\{D\},\qquad W_1\in L_1\setminus\{D,C\},\qquad W_2\in L_2\setminus\{D,B\},
\]
and auxiliary lines $\ell_i\in \mathcal P(W_i)\setminus\{n_i,L_i\}$, where
\[
n_0:=\overline{AW_0},\qquad n_1:=\overline{BW_1},\qquad n_2:=\overline{CW_2},
\]
so that the trimmed charts
\[
R_0:=R(n_0,W_0;\ell_0),\qquad R_1:=R(n_1,W_1;\ell_1),\qquad R_2:=R(n_2,W_2;\ell_2)
\]
satisfy:
\begin{enumerate}[leftmargin=2em]
  \item $(A,L_0)\in R_0$, $(B,L_1)\in R_1$, and $(C,L_2)\in R_2$;
  \item each pairwise overlap $R_i\cap R_j$ is nonempty;
  \item the chart adjacency graph on $\{R_0,R_1,R_2\}$ is a $3$-cycle.
\end{enumerate}
\end{lemma}

\begin{proof}
Choose $W_0\in L_0\setminus\{D\}$ arbitrarily. Then $A\in n_0\setminus\{W_0\}$ and $L_0\in \mathcal P(W_0)\setminus\{n_0\}$,
so after later choosing $\ell_0\neq L_0$ we will have $(A,L_0)\in R_0$.

Next choose $W_1\in L_1\setminus\{D,C\}$ so that $R_0\cap R_1\neq\varnothing$. The failure of this overlap is caused only when one of a
finite list of incidences holds: $n_1=\overline{BW_1}$ coincides with $\overline{W_0W_1}$, or $n_0$ coincides with $\overline{W_0W_1}$, or the
intersection point $p_{01}:=n_0\cap n_1$ lies on $\overline{W_0W_1}$, or $p_{01}\in\{W_0,W_1\}$. Each condition cuts out at most one point of
$L_1$, so together with the mandatory exclusions $\{D,C\}$ they remove at most six points. Since $|L_1|=q+1\ge 7$, a valid choice of $W_1$
exists. For such a choice,
\[
p_{01}:=n_0\cap n_1,
\qquad
\ell_{01}:=\overline{W_0W_1}
\]
gives one nonincidence edge $(p_{01},\ell_{01})\in R_0\cap R_1$.

Now choose $W_2\in L_2\setminus\{D,B\}$ so that simultaneously $R_1\cap R_2\neq\varnothing$ and $R_2\cap R_0\neq\varnothing$. Each of these
two overlap requirements excludes only finitely many points of $L_2$, again by the same ``one forbidden incidence = one forbidden point''
argument, so for $q\ge 6$ at least one admissible choice remains. For that choice there are nonincidence edges in both $R_1\cap R_2$ and
$R_2\cap R_0$.

Finally pick $\ell_i\in \mathcal P(W_i)\setminus\{n_i,L_i\}$ avoiding the finitely many overlap lines already selected for the nonempty
intersections. Then the trimmed charts still contain $(A,L_0)$, $(B,L_1)$, $(C,L_2)$, and all three pairwise overlaps remain nonempty. Hence
$R_0,R_1,R_2$ form a $3$-cycle in the chart adjacency graph.
\end{proof}

\begin{remark}[A minimal holonomy test configuration near $B_*$]
\label{rem:three-chart-cycle-holonomy-lab}
Under the hypotheses of \cref{lem:three-chart-cycle-bstar}, if the three charts are all $\Lam$-inactive then their pairwise overlaps define
transition scalars $c_{01},c_{12},c_{20}\in k^*$ and hence a concrete chart holonomy
\[
\operatorname{Hol}_{\mathcal A}(R_0R_1R_2R_0)=c_{01}c_{12}c_{20}.
\]
This produces an explicit geometrically forced cycle of degenerate charts already touching one strict $B_*$ witness. The remaining task
is to rewrite that holonomy in directly observable residue ratios and couple it to the skew determinant isolated by
\cref{prop:bstar-isolates-skew-delta}.
\end{remark}

\begin{definition}[Canonical base-column gauge]
\label{def:canonical-base-column-gauge}
Let $R=R(n,W;\ell)$ be a $\Lam$-inactive trimmed degenerate chart, and choose a \emph{base column}
\[
m\in \mathcal P(W)\setminus\{n,\ell\}.
\]
We say that $R$ is \emph{canonically gauged by $m$} if there exist scalars $\beta_r\in k^*$ with $\beta_m=1$ such that
\[
u_{Z,r}=u_{Z,m}\,\beta_r
\qquad\text{for every }(Z,r)\in R.
\tag{22.1}
\]
This normalization always exists on a $\Lam$-inactive chart: if $u_{Z,r}=\alpha_Z\gamma_r$ is any residue factorization on $R$, then setting
$\beta_r:=\gamma_r/\gamma_m$ gives \textup{(22.1)}, and the condition $\beta_m=1$ makes the $\beta_r$ unique.
\end{definition}

\begin{lemma}[Transition scalar is observable under canonical gauges]
\label{lem:observable-transition-scalar}
Let $R_i=R(n_i,W_i;\ell_i)$ and $R_j=R(n_j,W_j;\ell_j)$ be two $\Lam$-inactive trimmed charts, canonically gauged by base columns
\[
m_i\in \mathcal P(W_i)\setminus\{n_i,\ell_i\},
\qquad
m_j\in \mathcal P(W_j)\setminus\{n_j,\ell_j\}.
\]
Assume that the row lines intersect at
\[
p:=n_i\cap n_j,
\qquad p\neq W_i,W_j,
\]
and that $p\notin m_i,m_j$, so that $u_{p,m_i},u_{p,m_j}\in k^*$. Then the transition scalar from chart $i$ to chart $j$ is
\[
c_{ij}=\frac{u_{p,m_j}}{u_{p,m_i}}\in k^*.
\tag{22.2}
\]
In particular, the transition is directly observable from residue labels and does not depend on hidden choices of $\alpha$- or $\beta$-factors.
\end{lemma}

\begin{proof}
In the canonical gauge of $R_i$, the row factor at $p$ is exactly $u_{p,m_i}$, because \textup{(22.1)} with $r=m_i$ gives
$u_{p,m_i}=u_{p,m_i}\beta_{m_i}$ and $\beta_{m_i}=1$. Likewise the row factor at $p$ in chart $j$ is $u_{p,m_j}$. By
\cref{lem:atlas-transition-scalar}, the unique overlap scalar $c_{ij}$ is the ratio of the two row factors at any overlap point, hence
\[
c_{ij}=\frac{u_{p,m_j}}{u_{p,m_i}}.
\]
\end{proof}

\begin{corollary}[Chart holonomy is an observable product of residue ratios]
\label{cor:observable-chart-holonomy}
Consider a cycle of $\Lam$-inactive canonically gauged charts
\[
R_0\to R_1\to \cdots \to R_{t-1}\to R_0,
\]
where chart $R_i$ has base column $m_i$, and write
\[
p_i:=n_i\cap n_{i+1}
\qquad (i\bmod t)
\]
for the intersection point of consecutive row lines. Assume $p_i\notin m_i,m_{i+1}$ for every $i$, so that the residues
$u_{p_i,m_i},u_{p_i,m_{i+1}}\in k^*$ are defined. Then the chart holonomy is
\[
\operatorname{Hol}_{\mathcal A}(R_0R_1\cdots R_{t-1}R_0)
=\prod_{i=0}^{t-1}\frac{u_{p_i,m_{i+1}}}{u_{p_i,m_i}}\in k^*.
\tag{22.3}
\]
\end{corollary}

\begin{proof}
Apply \cref{lem:observable-transition-scalar} to each consecutive pair and multiply the resulting transition scalars around the cycle.
\end{proof}

\begin{remark}[Observable chart holonomy]
\label{rem:observable-holonomy-curvature}
The product in \textup{(22.3)} is an explicit curvature candidate that can be computed without hidden gauge variables:
it is a monomial ratio in actual residue entries of the lifted incidence matrix. If
\[
\operatorname{Hol}_{\mathcal A}(R_0R_1\cdots R_{t-1}R_0)\neq 1,
\]
then the canonical gauges fail to close around the chart cycle. This is the present residue-atlas analogue of a holonomy defect.
\end{remark}

\begin{lemma}[Bridge-cycle holonomy equals the governing cross-ratio]
\label{lem:bridge-cycle-holonomy-rho}
Assume one is in a strict $B_*$ setup as in \cref{prop:bstar-isolates-skew-delta}, so in particular
\[
A,B\notin L_0,L_1.
\]
Choose two bridge lines $t,s$ such that:
\begin{enumerate}[leftmargin=2em,label=(\roman*)]
  \item $t$ meets $L_0$ at $W_0\neq D$ and meets $L_1$ at $W_1\neq D$, with $A\notin t$;
  \item $s$ meets $L_1$ at $V_1\neq D$ and meets $L_0$ at $V_0\neq D$, with $B\notin s$.
\end{enumerate}
Define four trimmed charts with canonical base columns
\[
R_0:=R(n_0,W_0;\ell_0),\quad m_0:=L_0,\qquad n_0:=\overline{AW_0},
\]
\[
R_1:=R(n_1,W_1;\ell_1),\quad m_1:=t,\qquad n_1:=\overline{AW_1},
\]
\[
R_2:=R(n_2,V_1;\ell_2),\quad m_2:=L_1,\qquad n_2:=\overline{BV_1},
\]
\[
R_3:=R(n_3,V_0;\ell_3),\quad m_3:=s,\qquad n_3:=\overline{BV_0},
\]
and assume that each chart is $\Lam$-inactive and that the next base column is not excluded, i.e.
\[
t\neq \ell_0,\qquad L_1\neq \ell_1,\qquad s\neq \ell_2,\qquad L_0\neq \ell_3.
\]
Then the chart holonomy around the $4$-cycle
\[
R_0\to R_1\to R_2\to R_3\to R_0
\]
is the observable cross-ratio
\[
\operatorname{Hol}_{\mathcal A}(R_0R_1R_2R_3R_0)
=\Bigl(\frac{u_{A,L_1}}{u_{A,L_0}}\Bigr)
 \Bigl(\frac{u_{B,L_0}}{u_{B,L_1}}\Bigr)
=:\rho.
\]
\end{lemma}

\begin{proof}
By construction, the overlap $R_0\cap R_1$ contains the row point $A$, and the consecutive base columns are $L_0$ and $t$. Hence
\cref{lem:observable-transition-scalar} gives
\[
c_{01}=\frac{u_{A,t}}{u_{A,L_0}}.
\]
Likewise,
\[
c_{12}=\frac{u_{A,L_1}}{u_{A,t}},
\qquad
c_{23}=\frac{u_{B,s}}{u_{B,L_1}},
\qquad
c_{30}=\frac{u_{B,L_0}}{u_{B,s}}.
\]
Multiplying the four transition scalars and telescoping the intermediate residues $u_{A,t}$ and $u_{B,s}$ yields
\[
\operatorname{Hol}_{\mathcal A}(R_0R_1R_2R_3R_0)
=c_{01}c_{12}c_{23}c_{30}
=\Bigl(\frac{u_{A,L_1}}{u_{A,L_0}}\Bigr)
 \Bigl(\frac{u_{B,L_0}}{u_{B,L_1}}\Bigr).
\]
\end{proof}

\begin{remark}[Holonomy coupled to the skew determinant]
The bridge-cycle holonomy is no longer merely an abstract atlas obstruction. By \cref{lem:bridge-cycle-holonomy-rho}, it is exactly the multiplicative defect that decides whether the skew determinant
\[
\Delta_{AB}^{L_0L_1}=u_{A,L_0}u_{B,L_1}-u_{A,L_1}u_{B,L_0}
\]
cancels. This is the point at which an observable chart holonomy becomes algebraically tied to the same residue bottleneck that appears in a strict $(0,3)$ witness.
\end{remark}

\begin{definition}[Cross-ratio on a $2\times2$ zero rectangle]
\label{def:cross-ratio-rectangle}
Let $p,q$ be two point-vertices and $\ell,m$ two line-vertices in a zero-graph, and assume that
\[(p,\ell),(p,m),(q,\ell),(q,m)\in E(G_0(B)).\]
Define the associated \emph{cross-ratio} by
\[
\rho_{pq}^{\ell m}:=\frac{u_{p,m}}{u_{p,\ell}}\cdot\frac{u_{q,\ell}}{u_{q,m}}\in k^*.
\tag{24.1}
\]
Under a row/column rescaling $u_{x,y}\mapsto \alpha_x\beta_y u_{x,y}$, the value $\rho_{pq}^{\ell m}$ is unchanged. Thus it is a gauge-invariant residue observable of the $2\times2$ zero rectangle.
\end{definition}

\begin{lemma}[$\Lam$-inactive rectangle implies trivial cross-ratio]
\label{lem:lam-inactive-implies-rho1}
Assume that on the row set $\{p,q\}$ and column set $\{\ell,m\}$ one has a rank-$1$ factorization
\[u_{x,y}=\alpha_x\beta_y
\qquad(x\in\{p,q\},\ y\in\{\ell,m\})\]
in the residue labels. Then
\[
\rho_{pq}^{\ell m}=1.
\tag{24.2}
\]
In particular, every $\Lam$-inactive $2\times2$ zero rectangle has trivial cross-ratio.
\end{lemma}

\begin{proof}
Substituting $u_{x,y}=\alpha_x\beta_y$ into \textup{(24.1)} gives
\[
\rho_{pq}^{\ell m}=\frac{\alpha_p\beta_m}{\alpha_p\beta_\ell}\cdot\frac{\alpha_q\beta_\ell}{\alpha_q\beta_m}=1.
\]
\end{proof}

\begin{remark}[Cross-ratios as local obstruction units]
\label{rem:cross-ratio-atomic-obstruction}
The quantity $\rho_{pq}^{\ell m}$ is the smallest gauge-invariant scalar that distinguishes ``already pure gauge'' from ``cancellation-sensitive'' behaviour on one valuation-$0$ rectangle. By \cref{lem:lam-inactive-implies-rho1}, any rectangle with $\rho\neq1$ cannot be explained by a local rank-$1$ residue chart. This provides a convenient countable local obstruction for the present analysis.
\end{remark}

\begin{lemma}[Inversion and complement moves on skew cross-ratios]
\label{lem:cross-ratio-inversion-complement}
Let $R=(p,q;\ell,m)$ be a skew all-zero $2\times2$ rectangle, and write
\[
\rho(R):=\rho_{pq}^{\ell m}=\frac{u_{p,m}}{u_{p,\ell}}\cdot\frac{u_{q,\ell}}{u_{q,m}}\in k^*.
\]
Then:
\begin{enumerate}[leftmargin=2em]
  \item swapping the two points or the two lines inverts the cross-ratio,
  \[
  \rho_{qp}^{\ell m}=\rho_{pq}^{m\ell}=\rho(R)^{-1};
  \]
  \item whenever $(A,B,C,D;L_0,L_1,L_2,L_3)$ is a strict $B_*$ witness, the two skew rectangles
  \[
  R_1:=(A,B;L_0,L_1),\qquad R_2:=(A,C;L_0,L_2)
  \]
  satisfy
  \[
  \rho(R_2)=1-\rho(R_1).
  \]
\end{enumerate}
Equivalently, the basic residue moves are the involutions
\[
I(x)=1/x,\qquad S(x)=1-x.
\]
\end{lemma}

\begin{proof}
The inversion identities are immediate from \cref{def:cross-ratio-rectangle} by exchanging the roles of the two rows or the two columns.
For the complement identity, \cref{lem:bstar-rho-plus-sigma} gives
\[
\rho_{AB}^{L_0L_1}+\rho_{AC}^{L_0L_2}=1,
\]
which is exactly $\rho(R_2)=1-\rho(R_1)$.
\end{proof}



\section{Identity-pattern minors and local determinant identities}
\label{sec:gadgets}

\subsection{\texorpdfstring{Identity-pattern $4\times4$ minors: leading derangement equation}{Identity-pattern 4x4 minors: leading derangement equation}}

\begin{definition}[Identity valuation pattern for $4\times4$ blocks]
\label{def:identity4-pattern}
We use the $4\times4$ identity valuation pattern
\[
A^{(4)}=\begin{pmatrix}
1&0&0&0\\
0&1&0&0\\
0&0&1&0\\
0&0&0&1
\end{pmatrix},
\]
i.e.\ diagonal entries have valuation $1$ and off-diagonal entries have valuation $0$.
This is the $k=4$ instance of \cref{def:identity-valuation-pattern}.
\end{definition}

\begin{proposition}[Leading derangement equation for a $4\times4$ identity block]
\label{prop:identity4-theta-equation}
Let $F=(f_{ij})$ be a lift of $A^{(4)}$ over a valued field with residue field $k$, and write
\[
f_{ij}=u_{ij}+\text{(higher-valuation terms)}\qquad (i\neq j),
\]
while $\val(f_{ii})=1$ on the diagonal.
Define
\[
\Theta_4(U):=\sum_{\sigma\in D_4}\operatorname{sgn}(\sigma)\prod_{i=1}^4 u_{i,\sigma(i)},
\]
where $D_4\subset S_4$ is the set of derangements.
Then
\[
\det(F)=\Theta_4(U)+\text{(terms of valuation }>0),
\]
so
\[
\val(\det F)>0
\qquad\Longleftrightarrow\qquad
\Theta_4(U)=0.
\]
In particular, every rank-$\le 3$ lift must solve one $9$-term residue cancellation equation on each $4\times4$ identity-pattern minor.
\end{proposition}

\begin{proof}
This is the $k=4$ specialization of \cref{lem:leading-term-minor-test}. For the valuation pattern $A^{(4)}$, each fixed point of a permutation contributes $1$ to the tropical weight and each moved point contributes $0$. Hence the valuation-minimal permutations are exactly the derangements, and there are $!4=9$ of them. The displayed initial form is therefore precisely the signed derangement sum over the off-diagonal residues. The equivalence
\[
\val(\det F)>0 \iff \Theta_4(U)=0
\]
follows immediately from \cref{lem:leading-term-minor-test}.
\end{proof}

\begin{remark}[Rank-$1$ residue charts are excluded]
\label{rem:identity4-rank1-excluded}
By \cref{cor:identity4-rank1-obstruction}, the equation $\Theta_4(U)=0$ cannot hold when the off-diagonal residue data factorizes rank $1$ and the residue characteristic is not $3$. Thus every vanished $4\times4$ identity block already forces non-factorized residue behavior.
\end{remark}

\begin{definition}[Admissible residue cross-ratios inside an identity block]
\label{def:identity4-admissible-crossratio}
Let $U=(u_{ij})_{i\neq j}$ be the off-diagonal residue data of a $4\times4$ identity-pattern block.
For distinct rows $i\neq j$ and distinct columns $a\neq b$ with
\[
\{i,j\}\cap\{a,b\}=\varnothing,
\]
we define the associated \emph{admissible cross-ratio}
\[
\rho_{ij}^{ab}:=\frac{u_{ia}u_{jb}}{u_{ib}u_{ja}}\in k^*.
\]
Equivalently, $\rho_{ij}^{ab}=1$ if and only if
\[
u_{ia}u_{jb}=u_{ib}u_{ja}.
\]
These are exactly the $2\times2$ off-diagonal rectangles whose four entries are defined inside the identity-pattern block.
\end{definition}

\begin{figure}[htbp]
\centering
\begin{tikzpicture}[font=\small]
  \matrix (m) [matrix of nodes,
               nodes={draw,minimum width=9mm,minimum height=8mm,anchor=center,font=\small},
               row sep=-\pgflinewidth,column sep=-\pgflinewidth] {
      1 & 0 & 0 & 0 \\
      0 & 1 & 0 & 0 \\
      0 & 0 & 1 & 0 \\
      0 & 0 & 0 & 1 \\
  };
  \draw[dashed,line width=.8pt] ($(m-1-3.north west)+(-2pt,2pt)$) rectangle ($(m-2-4.south east)+(2pt,-2pt)$);
  \node[above=0.8em of m] {$A^{(4)}$};
  \node[below=1.2em of m, align=center] {The dashed rectangle is an admissible off-diagonal $2\times2$ subblock\\ associated with $\rho_{12}^{34}=\dfrac{u_{13}u_{24}}{u_{14}u_{23}}$.};
\end{tikzpicture}
\caption{The $4\times4$ identity valuation pattern. Diagonal entries carry valuation $1$, while every off-diagonal entry has valuation $0$.}
\label{fig:identity-pattern}
\end{figure}

\begin{lemma}[Cross-ratio flatness implies rank-$1$ factorization off the diagonal]
\label{lem:identity-flatness-rank1}
Let $n\ge 4$, and let $(u_{ij})$ be nonzero scalars in a field $k$, defined for all ordered pairs $i\neq j$ in $\{1,\dots,n\}$.
Assume that for every distinct rows $i\neq j$ and distinct columns $a\neq b$ with $\{i,j\}\cap\{a,b\}=\varnothing$, one has
\[
 u_{ia}u_{jb}=u_{ib}u_{ja}.
\tag{4.2}
\]
Then there exist scalars $\alpha_1,\dots,\alpha_n\in k^*$ and $\beta_1,\dots,\beta_n\in k^*$ such that
\[
 u_{ij}=\alpha_i\beta_j
 \qquad\text{for all } i\neq j.
\tag{4.3}
\]
\end{lemma}

\begin{proof}
Fix row $1$, and set $\beta_a:=u_{1a}$ for each $a\neq 1$.
Now fix $i\neq 1$.
Because $n\ge 4$, there exist at least two columns $a,b$ with $a,b\notin\{1,i\}$ and $a\neq b$.
For any such $a$, define
\[
\alpha_i^{(a)}:=\frac{u_{ia}}{u_{1a}}.
\]
If $a,b\notin\{1,i\}$ are distinct, then \textup{(4.2)} applied to rows $(i,1)$ and columns $(a,b)$ gives
\[
 u_{ia}u_{1b}=u_{ib}u_{1a}.
\]
Dividing by $u_{1a}u_{1b}$ yields
\[
 \frac{u_{ia}}{u_{1a}}=\frac{u_{ib}}{u_{1b}},
\]
so $\alpha_i^{(a)}$ is independent of $a$.
We therefore write the common value as $\alpha_i$.

Finally, let $i\neq j$ be arbitrary.
Choose $a\notin\{1,i,j\}$, which is possible because $n\ge 4$.
Applying \textup{(4.2)} to rows $(i,1)$ and columns $(j,a)$ gives
\[
 u_{ij}u_{1a}=u_{ia}u_{1j}.
\]
Hence
\[
 u_{ij}=\frac{u_{ia}}{u_{1a}}\,u_{1j}=\alpha_i\beta_j.
\]
This proves \textup{(4.3)}.
\end{proof}

\begin{corollary}[Vanishing identity blocks force a cross-ratio defect]
\label{cor:identity4-vanishing-forces-curvature}
Let $F$ be a lift of a $4\times4$ identity-pattern block over a valued field with residue field $k$ of characteristic $\neq 3$, and let $U=(u_{ij})_{i\neq j}$ denote its off-diagonal residue data.
Assume that every admissible cross-ratio from Definition~\ref{def:identity4-admissible-crossratio} is equal to $1$.
Then $\det(F)\neq 0$.
Equivalently, if $\det(F)=0$, then there exists at least one admissible rectangle with
\[
\rho_{ij}^{ab}\neq 1.
\]
\end{corollary}

\begin{proof}
If every admissible cross-ratio is $1$, then the hypotheses of \cref{lem:identity-flatness-rank1} hold with $n=4$, so the off-diagonal residues factor as $u_{ij}=\alpha_i\beta_j$ for all $i\neq j$.
By \cref{cor:identity4-rank1-obstruction}, such rank-$1$ residue data cannot satisfy $\Theta_4(U)=0$ in residue characteristic $\neq 3$.
Therefore \cref{prop:identity4-theta-equation} implies $\val(\det F)=0$, hence $\det(F)\neq 0$.
The contrapositive gives the final statement.
\end{proof}

\begin{remark}[Cross-ratio-defect formulation]
\label{rem:identity4-curvature-witness}
\Cref{cor:identity4-vanishing-forces-curvature} is the first local charging statement for the identity-block analysis: every vanished $4\times4$ identity-pattern minor must contain at least one residue rectangle with nontrivial cross-ratio.
So the obstruction is no longer only ``not rank $1$'' in the abstract; it is witnessed by a concrete cross-ratio defect $\rho\neq 1$ inside the block.
\end{remark}


\subsection{Residue-rank reduction and the first-order deformation bottleneck}

\begin{lemma}[Reduction does not increase rank]
\label{lem:reduction-does-not-increase-rank}
Let $F\in R_K^{m\times n}$, and let $\pi:R_K\to k=R_K/\mathfrak m_K$ be the residue map.
Then
\[
\rank_k(\pi(F))\le \rank_K(F).
\]
\end{lemma}

\begin{proof}
If $\rank_K(F)\le r$, then every $(r+1)\times(r+1)$ minor determinant of $F$ is zero.
Determinant is a polynomial in the matrix entries with integer coefficients, so reduction commutes with determinant evaluation:
\[
\det\bigl(\pi(F)_{I,J}\bigr)=\pi\bigl(\det(F_{I,J})\bigr)=\pi(0)=0
\]
for every choice of $(r+1)$ rows $I$ and $(r+1)$ columns $J$.
Hence every $(r+1)\times(r+1)$ minor of $\pi(F)$ vanishes, proving $\rank_k(\pi(F))\le r$.
\end{proof}

\begin{corollary}[Low Kapranov rank gives a rank-$3$ residue zero-pattern model]
\label{cor:kapr3-residue-zero-pattern}
Let $\Pi$ be a projective plane with incidence matrix $I(\Pi)$.
Suppose $F$ is a rank-$\le 3$ lift of $I(\Pi)$, so
\[
\val(F_{p\ell})=
\begin{cases}
1,& p\in \ell,\\
0,& p\notin \ell.
\end{cases}
\]
Then the residue matrix
\[
U:=\pi(F)\in k^{P\times L}
\]
satisfies
\[
\rank_k(U)\le 3,
\]
and
\[
U_{p\ell}=0 \iff p\in \ell,
\qquad
U_{p\ell}\neq 0 \iff p\notin \ell.
\]
\end{corollary}

\begin{proof}
All entries of $F$ lie in $R_K$ because their valuations are $0$ or $1$.
So the rank bound follows from \cref{lem:reduction-does-not-increase-rank}.
If $p\in\ell$, then $\val(F_{p\ell})=1$, hence $F_{p\ell}\in\mathfrak m_K$ and $U_{p\ell}=0$.
If $p\notin\ell$, then $\val(F_{p\ell})=0$, so $F_{p\ell}$ is a unit in $R_K$ and its residue is nonzero.
\end{proof}

\begin{proposition}[First-order deformation directions for a rank-$3$ residue chart]
\label{prop:first-order-tangent-directions}
Assume $U\in k^{m\times n}$ has rank exactly $3$ and factors as
\[
U=A_0B_0,
\qquad
A_0\in k^{m\times 3},\ B_0\in k^{3\times n}.
\]
Fix a uniformizer $t$ and consider first-order deformations
\[
A(t)=A_0+tA_1,
\qquad
B(t)=B_0+tB_1.
\]
Then
\[
F(t):=A(t)B(t)=U+tV+O(t^2),
\qquad
V=A_1B_0+A_0B_1.
\tag{4.4}
\]
In particular, the allowable first-order directions of rank-$3$ deformations of $U$ form the linear space
\[
T_U:=\{A_1B_0+A_0B_1:\ A_1\in k^{m\times 3},\ B_1\in k^{3\times n}\}.
\tag{4.5}
\]
\end{proposition}

\begin{proof}
Expand the product:
\[
(A_0+tA_1)(B_0+tB_1)=A_0B_0+t(A_1B_0+A_0B_1)+t^2A_1B_1.
\]
Since $A_0B_0=U$, the claimed formula follows immediately.
The description of $T_U$ is just the set of all possible first-order coefficients $V$ arising from such factorizations.
\end{proof}

\begin{remark}[Incidence valuation $1$ is a first-order nonvanishing condition]
\label{rem:incidence-valuation-one-first-order}
Under \cref{cor:kapr3-residue-zero-pattern}, any uniformizer expansion of a rank-$\le 3$ lift has the form
\[
F_{p\ell}(t)=U_{p\ell}+tV_{p\ell}+O(t^2).
\]
For an incidence entry one needs valuation exactly $1$, so necessarily
\[
U_{p\ell}=0,
\qquad
V_{p\ell}\neq 0
\qquad
\text{for every }p\in\ell.
\tag{4.6}
\]
Thus the hard part is not merely the existence of a rank-$3$ residue matrix with the projective-plane zero pattern, but the existence of one tangent direction $V\in T_U$ that is nonzero on all incidence positions.
\end{remark}

\begin{corollary}[Rank bound for first-order coefficients]
\label{cor:first-order-coefficient-rank-bound}
Let $F(t)\in K^{m\times n}$ have classical rank at most $3$, and write
\[
F(t)=U+tV+O(t^2)
\]
entrywise in one uniformizer $t$.
Then
\[
\rank_k(V)\le 6.
\]
\end{corollary}

\begin{proof}
Choose a rank factorization
\[
F(t)=A(t)B(t),
\qquad
A(t)\in K^{m\times 3},\ B(t)\in K^{3\times n}.
\]
Expanding
\[
A(t)=A_0+tA_1+O(t^2),
\qquad
B(t)=B_0+tB_1+O(t^2),
\]
we get
\[
V=A_1B_0+A_0B_1.
\]
Each summand has rank at most $3$, so $\rank_k(V)\le 6$.
\end{proof}

\begin{definition}[Monomial lift of the incidence matrix]
\label{def:monomial-lift-incidence}
Let $\Pi$ be a projective plane with incidence matrix $I(\Pi)$.
A rank-$\le 3$ lift of $I(\Pi)$ is called \emph{monomial} if, after choosing one uniformizer $t$, it has the form
\[
F_{p\ell}(t)=
\begin{cases}
 t a_{p\ell}, & p\in\ell,\\
 u_{p\ell}, & p\notin\ell,
\end{cases}
\qquad
 a_{p\ell},u_{p\ell}\in k^*.
\]
Equivalently,
\[
F(t)=U+tA,
\]
where $U_{p\ell}=0$ exactly on incidences and $A_{p\ell}=0$ exactly on nonincidences.
\end{definition}

\begin{definition}[First-order expansion data on one identity block]
\label{def:identity4-first-order-expansion}
Let $M(t)=(m_{ij}(t))$ be a $4\times4$ minor with identity valuation pattern.
We write
\[
 m_{ii}(t)=t a_i+O(t^2)
 \qquad (a_i\in k^*),
\]
and for $i\neq j$,
\[
 m_{ij}(t)=u_{ij}+t w_{ij}+O(t^2)
 \qquad (u_{ij}\in k^*,\; w_{ij}\in k).
\]
We denote the diagonal first coefficients by $A=(a_1,a_2,a_3,a_4)$, the off-diagonal residues by $U=(u_{ij})_{i\neq j}$, and the off-diagonal first corrections by $W=(w_{ij})_{i\neq j}$.
For each $i\in\{1,2,3,4\}$, if $\{j,k,\ell\}=\{1,2,3,4\}\setminus\{i\}$, define
\[
D_i(U):=u_{jk}u_{k\ell}u_{\ell j}+u_{j\ell}u_{\ell k}u_{kj}.
\]
Thus $D_i(U)$ is the contribution of the two $3$-cycles on the remaining three indices.
\end{definition}

\begin{lemma}[Determinant expansion to first order on an identity block]
\label{lem:identity4-determinant-first-order}
With the notation of Definition~\ref{def:identity4-first-order-expansion}, one has
\[
\det M(t)=\Theta_4(U)+t\,\Psi_4(U,A,W)+O(t^2),
\]
where
\[
\Psi_4(U,A,W)
:=
\sum_{i=1}^4 a_i D_i(U)
+
\sum_{\sigma\in D_4}\operatorname{sgn}(\sigma)
\Bigl(\prod_{r=1}^4 u_{r,\sigma(r)}\Bigr)
\Bigl(\sum_{r=1}^4 \frac{w_{r,\sigma(r)}}{u_{r,\sigma(r)}}\Bigr).
\]
\end{lemma}

\begin{proof}
Expand $\det M(t)$ by the Leibniz formula.
The valuation-$0$ terms are exactly the derangements, since every fixed point contributes one diagonal factor of valuation $1$.
This gives the leading coefficient $\Theta_4(U)$.

For the valuation-$1$ coefficient there are two possibilities.
First, one may choose a permutation with exactly one fixed point $i$; then the remaining three indices are arranged by one of the two $3$-cycles, each of sign $+1$, producing the contribution $a_iD_i(U)$.
Second, one may choose a derangement and replace exactly one off-diagonal factor $u_{r,\sigma(r)}$ by its first-order correction $w_{r,\sigma(r)}$.
Summing these contributions over all derangements gives the displayed formula for $\Psi_4(U,A,W)$.
All other Leibniz terms have valuation at least $2$, so they are absorbed into $O(t^2)$.
\end{proof}

\begin{corollary}[Each vanished identity block forces two consecutive equations]
\label{cor:identity4-two-equation-package}
If $M(t)$ is a $4\times4$ identity-pattern minor of a rank-$\le 3$ lift, then $\det M(t)=0$ and therefore
\[
\Theta_4(U)=0,
\qquad
\Psi_4(U,A,W)=0.
\]
So every such block imposes not only the leading derangement cancellation equation, but also one first-order compatibility equation coupling the incidence coefficients $a_i$ and the off-incidence first corrections $w_{ij}$.
\end{corollary}

\begin{proof}
A rank-$\le 3$ lift has vanishing determinant on every $4\times4$ minor.
Applying Lemma~\ref{lem:identity4-determinant-first-order} and comparing the coefficients of $t^0$ and $t^1$ gives the claim.
\end{proof}

\begin{remark}[Monomial specialization of the first-order equation]
\label{rem:identity4-monomial-specialization}
If the off-diagonal entries of the chosen identity block carry no valuation-$1$ correction, i.e.
\[
 w_{ij}=0\qquad(i\neq j),
\]
then
\[
\Psi_4(U,A,0)=\sum_{i=1}^4 a_iD_i(U).
\]
Thus any such \emph{monomial} block must satisfy one linear relation among its four incidence first coefficients after the leading equation $\Theta_4(U)=0$ has already been enforced.
\end{remark}

\begin{remark}[Location of the first-order obstruction]
\label{rem:first-order-bottleneck}
The residue matrix $U$ already satisfies many polynomial constraints: each identity-pattern block forces one equation $\Theta_4(U)=0$, and globally any rank-$\le 3$ lift must support $\Omega(q^8)$ cross-ratio-defective rectangles. The next obstruction layer is the search for a tangent direction in $T_U$ that is nonzero on every incidence. The analysis above makes that layer algebraically explicit: on each identity block the same lift must satisfy both $\Theta_4(U)=0$ and the first-order equation $\Psi_4(U,A,W)=0$.
Any forced vanishing $V_{p\ell}=0$ at an incidence would raise that entry to valuation at least $2$, which is exactly the kind of extra stored cancellation depth encoded by the $\Lam$-layer.
\end{remark}


\subsection{Abundance of identity minors in projective-plane incidence matrices}

\begin{definition}[Ordered quadrangle]
\label{def:ordered-quadrangle}
Let $\Pi$ be a projective plane. An ordered quadruple $(P_1,P_2,P_3,P_4)$ of points is an \emph{ordered quadrangle} if no three of the four points are collinear.
\end{definition}

\begin{lemma}[Private-line identity blocks]
\label{lem:private-line-identity-minors}
Let $\Pi$ be a projective plane of order $q\ge 3$ with combinatorial incidence matrix $I(\Pi)$.
Fix an ordered quadrangle $(P_1,P_2,P_3,P_4)$.
Then for each $i\in\{1,2,3,4\}$ there are exactly $q-2$ lines $L_i$ such that
\[
P_i\in L_i,
\qquad
P_j\notin L_i\quad (j\neq i).
\]
Choosing one such $L_i$ for each $i$ produces a $4\times4$ submatrix of $I(\Pi)$, with rows $(P_1,\dots,P_4)$ and columns $(L_1,\dots,L_4)$, equal to $I_4$.
Equivalently, that block has valuation pattern $A^{(4)}$.
\end{lemma}

\begin{proof}
Fix $i$. Exactly $q+1$ lines pass through $P_i$.
For each $j\neq i$, there is a unique line through $P_i$ and $P_j$, namely $\overline{P_iP_j}$.
Because $(P_1,P_2,P_3,P_4)$ is a quadrangle, the three lines $\overline{P_iP_j}$ for $j\neq i$ are distinct.
Any line through $P_i$ that also contains some other $P_j$ must be one of those three forbidden lines.
Hence the number of lines through $P_i$ that avoid the other three points is
\[
(q+1)-3=q-2.
\]
If we choose one such line $L_i$ for each $i$, then $P_i\in L_i$ by construction, while $P_j\notin L_i$ for $j\neq i$.
So the induced $4\times4$ incidence submatrix is exactly $I_4$.
\end{proof}

\begin{proposition}[Quantitative lower bound for identity-pattern minors]
\label{prop:identity-minor-count-lower-bound}
Let $\Pi$ be a projective plane of order $q\ge 3$, and set $v:=q^2+q+1$.
Let $N_{\mathrm{id}}^{\mathrm{ord}}(\Pi)$ denote the number of ordered $4\times4$ submatrices of $I(\Pi)$ equal to $I_4$ and obtained from \cref{lem:private-line-identity-minors}.
Then
\[
N_{\mathrm{id}}^{\mathrm{ord}}(\Pi)\ge v(v-1)q^2(q-1)^2(q-2)^4.
\tag{4.1}
\]
In particular,
\[
N_{\mathrm{id}}^{\mathrm{ord}}(\Pi)=\Omega(q^{12}).
\]
After dividing by $(4!)^2$, the same construction still yields $\Omega(q^{12})$ distinct unordered $4\times4$ identity-pattern minors.
\end{proposition}

\begin{proof}
Choose $P_1$ arbitrarily: $v$ choices.
Choose $P_2\neq P_1$: $v-1$ choices.
Now choose $P_3$ off the unique line $\overline{P_1P_2}$.
Since each line has $q+1$ points, this gives
\[
v-(q+1)=q^2
\]
choices.

At this stage $P_1,P_2,P_3$ are noncollinear.
For $P_4$, exclude the three lines
\[
\overline{P_1P_2},\qquad \overline{P_1P_3},\qquad \overline{P_2P_3}.
\]
Each has $q+1$ points, and because $P_1,P_2,P_3$ are noncollinear, the three lines meet pairwise exactly at $P_1,P_2,P_3$.
So their union has size
\[
3(q+1)-3=3q.
\]
Hence the number of valid choices for $P_4$ is at least
\[
v-3q=(q^2+q+1)-3q=(q-1)^2.
\]
Any such choice makes $(P_1,P_2,P_3,P_4)$ an ordered quadrangle.

Finally, \cref{lem:private-line-identity-minors} gives exactly $q-2$ choices of $L_i$ for each $i=1,2,3,4$, contributing a factor $(q-2)^4$.
Multiplying the choices yields
\[
N_{\mathrm{id}}^{\mathrm{ord}}(\Pi)\ge v(v-1)\cdot q^2\cdot (q-1)^2\cdot (q-2)^4,
\]
which is \textup{(4.1)}.
Since $v=q^2+q+1=\Theta(q^2)$, the bound is $\Omega(q^{12})$.
The unordered estimate follows by dividing by at most $(4!)^2$ orderings of the rows and columns.
\end{proof}

\begin{proposition}[Abundant identity blocks force multi-term residue cancellation]
\label{prop:identity-minors-pressure}
Assume $\Kapr(I(\Pi))\le 3$ over a valued field whose residue field has characteristic $\neq 3$.
Then every ordered identity-pattern minor counted in \cref{prop:identity-minor-count-lower-bound} must satisfy one $9$-term leading cancellation equation
\[
\Theta_4(U)=0.
\]
Moreover, by \cref{cor:identity4-rank1-obstruction}, the off-diagonal residues on each such block cannot factorize rank $1$.
Consequently any rank-$\le 3$ lift of $I(\Pi)$ must accommodate $\Omega(q^{12})$ multi-term residue cancellations.
\end{proposition}

\begin{proof}
If $\Kapr(I(\Pi))\le 3$, there exists a lift of $I(\Pi)$ with classical rank at most $3$.
Hence every $4\times4$ minor of that lift vanishes.
For each identity-pattern block from \cref{prop:identity-minor-count-lower-bound}, \cref{prop:identity4-theta-equation} therefore forces the residue equation $\Theta_4(U)=0$.
The non-factorization claim follows from \cref{cor:identity4-rank1-obstruction}.
Finally, \cref{prop:identity-minor-count-lower-bound} provides the $\Omega(q^{12})$ supply of such blocks.
\end{proof}

\begin{remark}[Scope and next bottleneck]
\label{rem:identity-minors-scope-bottleneck}
The private-line construction starts only at $q\ge 3$; for $q=2$ one has $q-2=0$, so no such blocks exist.
Also, \cref{prop:identity-minor-count-lower-bound} is a supply statement rather than an independence statement: many of the counted minors share rows, columns, and residue variables.
The first charging step is to show that each fixed admissible $0$-rectangle lies in at most $O(q^4)$ ordered identity-pattern minors, so any rank-$\le 3$ lift in residue characteristic $\neq 3$ must support $\Omega(q^8)$ distinct cross-ratio defects.
The next bottleneck is therefore no longer to produce many witnesses, but to prove an upper bound on how many such cross-ratio defects a rank-$3$ lift can sustain, or else to show that forcing so many defects creates positive $\Lam$-depth on a large set of rectangles.
\end{remark}



\section{Obstruction consequences}
\label{sec:main}

\subsection{Counting cross-ratio-defective rectangles}

\begin{definition}[$0$-rectangles and non-flat residues]
\label{def:identity-zerorectangle-curvature-set}
Let $\Pi$ be a projective plane of order $q\ge 3$, with point set $P$, line set $L$, and incidence matrix $I(\Pi)$.
A \emph{$0$-rectangle} is a pair
\[
R=(\{p,q\}\times \{\ell,m\}),
\]
where $p\neq q$ are points, $\ell\neq m$ are lines, and
\[
I(\Pi)_{p\ell}=I(\Pi)_{pm}=I(\Pi)_{q\ell}=I(\Pi)_{qm}=0.
\]
Equivalently, $p$ and $q$ are both outside $\ell\cup m$.
For a lift $F$ of $I(\Pi)$, with off-incidence residues $u_{x,n}\in k^*$ at the $0$-entries, define
\[
\rho_F(R):=\frac{u_{p\ell}u_{qm}}{u_{pm}u_{q\ell}}\in k^*.
\]
The associated \emph{defect set} is
\[
\mathcal R_{\mathrm{def}}(F):=\{R:\rho_F(R)\neq 1\}.
\]
\end{definition}

\begin{lemma}[Multiplicity bound for one $0$-rectangle]
\label{lem:identity-rectangle-multiplicity}
Let $\Pi$ be a projective plane of order $q\ge 3$, and let
\[
R=(\{p,q\}\times \{\ell,m\})
\]
be a $0$-rectangle in $I(\Pi)$.
Let $\mathcal M_{\mathrm{id}}^{\mathrm{ord}}(\Pi)$ be the family of ordered identity-pattern minors counted in \cref{prop:identity-minor-count-lower-bound}.
Then
\[
\#\{M\in \mathcal M_{\mathrm{id}}^{\mathrm{ord}}(\Pi):R\subset M\}
\le C(q+1)^4
\tag{5.1}
\]
for an absolute constant $C$.
One may take, for example, $C=16(4!)^2$.
\end{lemma}

\begin{proof}
In any ordered $4\times4$ identity-pattern minor, the admissible $0$-rectangles are exactly the complement rectangles obtained by choosing two of the four point rows and the two line columns matched to the other two points. Hence a fixed rectangle $R$ can only occur when $\ell,m$ are the private lines of the two complementary vertices, say $r\in \ell$ and $s\in m$.
There are at most $(q+1)^2$ choices for $(r,s)$, since each of $\ell$ and $m$ contains exactly $q+1$ points.
Once $(p,q,r,s)$ are fixed, the private-line construction gives at most $(q+1)^2$ choices for the private lines assigned to $p$ and $q$.
Finally, only constantly many orderings of the four rows and four columns realize the prescribed rectangle as one of the six complement rectangles; this constant is absorbed into $C$.
Therefore the number of ordered identity-pattern minors containing $R$ is at most $C(q+1)^4$.
A detailed counting proof appears in \cref{lem:identity-rectangle-multiplicity-detailed}.
\end{proof}

\begin{theorem}[Rank-$3$ lifts force $\Omega(q^8)$ cross-ratio defects]
\label{thm:identity-global-curvature-pressure}
Let $\Pi$ be a projective plane of order $q\ge 3$, and assume the residue characteristic is not $3$.
If $\Kapr(I(\Pi))\le 3$, then every rank-$\le 3$ lift $F$ of $I(\Pi)$ satisfies
\[
|\mathcal R_{\mathrm{def}}(F)|
\ge
\frac{N_{\mathrm{id}}^{\mathrm{ord}}(\Pi)}{C(q+1)^4},
\tag{5.2}
\]
where $N_{\mathrm{id}}^{\mathrm{ord}}(\Pi)$ is the ordered identity-minor count from \cref{prop:identity-minor-count-lower-bound}.
In particular,
\[
|\mathcal R_{\mathrm{def}}(F)|=\Omega(q^8).
\]
\end{theorem}

\begin{proof}
Let $\mathcal M_{\mathrm{id}}^{\mathrm{ord}}(\Pi)$ denote the family of ordered identity-pattern minors supplied by \cref{prop:identity-minor-count-lower-bound}.
If $\Kapr(I(\Pi))\le 3$, then there exists a lift $F$ of rank at most $3$, so every $4\times4$ minor of $F$ vanishes.
In particular, every member of $\mathcal M_{\mathrm{id}}^{\mathrm{ord}}(\Pi)$ vanishes.
By \cref{cor:identity4-vanishing-forces-curvature}, each such vanished identity block contains at least one admissible rectangle with cross-ratio different from $1$.
Hence the defect set $\mathcal R_{\mathrm{def}}(F)$ is a hitting set for $\mathcal M_{\mathrm{id}}^{\mathrm{ord}}(\Pi)$.
By \cref{lem:identity-rectangle-multiplicity}, each rectangle in $\mathcal R_{\mathrm{def}}(F)$ belongs to at most $C(q+1)^4$ members of $\mathcal M_{\mathrm{id}}^{\mathrm{ord}}(\Pi)$.
Double counting the incidence relation
\[
\{(R,M):R\in \mathcal R_{\mathrm{def}}(F),\ M\in \mathcal M_{\mathrm{id}}^{\mathrm{ord}}(\Pi),\ R\subset M\}
\]
gives
\[
N_{\mathrm{id}}^{\mathrm{ord}}(\Pi)
\le
|\mathcal R_{\mathrm{def}}(F)|\,C(q+1)^4,
\]
which rearranges to \textup{(5.2)}.
Finally, \cref{prop:identity-minor-count-lower-bound} gives
\[
N_{\mathrm{id}}^{\mathrm{ord}}(\Pi)=\Omega(q^{12}),
\]
so dividing by $C(q+1)^4=O(q^4)$ yields $|\mathcal R_{\mathrm{def}}(F)|=\Omega(q^8)$.
\end{proof}

\begin{corollary}[Quantitative non-flatness requirement]
\label{cor:identity-global-curvature-pressure}
Under the hypotheses of \cref{thm:identity-global-curvature-pressure}, a rank-$\le 3$ lift of $I(\Pi)$ must be residue-nonflat on at least $\Omega(q^8)$ admissible $0$-rectangles.
Equivalently, low Kapranov rank forces a polynomial-size support of local non-rank-$1$ residue defects.
\end{corollary}

\begin{proof}
This is just a restatement of \cref{thm:identity-global-curvature-pressure} in the language of residue flatness versus cross-ratio defects.
\end{proof}


\subsection{Immediate local corollaries}
\begin{theorem}[Residue flatness is incompatible with a vanished identity block]
\label{thm:flat-chart-obstruction}
Let $F$ be a lift of a $4\times4$ identity-pattern block over a valued field with residue field of characteristic $\neq 3$.
If all admissible residue cross-ratios in this block are equal to $1$, then the off-diagonal residues factor rank $1$, the leading derangement equation $\Theta_4(U)=0$ fails, and therefore $\det(F)\neq 0$.
Equivalently, every vanished identity-pattern block contains at least one admissible rectangle with nontrivial cross-ratio.
\end{theorem}

\begin{proof}
Immediate from Lemma~\ref{lem:identity-flatness-rank1} and Corollary~\ref{cor:identity4-vanishing-forces-curvature}.
\end{proof}

\begin{remark}
This branch isolates where ``pure gauge'' explanations (rank-$1$ residue charts) are insufficient:
cancellation must involve more than a single binomial, and its first visible witness is already one local cross-ratio defect $\rho\neq 1$.
\end{remark}


\subsection{Residue-rank reduction and first-order deformations}

\begin{proposition}[Residue-level reformulation of the rank-$3$ problem]
\label{prop:main-residue-minrank-reformulation}
Assume $\Kapr(I(\Pi))\le 3$.
Then there exists a matrix $U\in k^{P\times L}$ such that
\[
\rank_k(U)\le 3,
\qquad
U_{p\ell}=0 \iff p\in\ell,
\qquad
U_{p\ell}\neq 0 \iff p\notin\ell.
\]
Equivalently, the bipartite complement of the Levi graph of $\Pi$ admits a rank-$3$ matrix representation over the residue field.
\end{proposition}

\begin{proof}
This is exactly \cref{cor:kapr3-residue-zero-pattern} applied to a rank-$\le 3$ lift of $I(\Pi)$.
\end{proof}

The unresolved globalization step is formulated in \cref{conj:first-order-nonvanishing-obstruction}, after the proof-driven part of the paper.

\begin{proposition}[First-order constraints on an identity block]
\label{prop:note38-first-order-constraints}
Assume $\Kapr(I(\Pi))\le 3$, and let
\[
F(t)=U+tV+O(t^2)
\]
be a local uniformizer expansion of a rank-$\le 3$ lift near one $4\times4$ identity-pattern minor.
Then the restricted data on that block satisfy
\[
\Theta_4(U)=0,
\qquad
\Psi_4(U,A,W)=0.
\]
In particular, if the off-incidence entries of that block are monomial to first order, then
\[
\sum_{i=1}^4 a_iD_i(U)=0.
\]
\end{proposition}

\begin{proof}
This is exactly Corollary~\ref{cor:identity4-two-equation-package} together with Remark~\ref{rem:identity4-monomial-specialization}, applied to the chosen identity-pattern minor.
\end{proof}

\begin{lemma}[One line yields a $(q+1)\times(q+1)$ incidence-supported diagonal block]
\label{lem:one-line-incidence-diagonal-block}
Let $\Pi$ be a projective plane of order $q$, and fix one line $\ell^*$ with points
\[
P_1,\dots,P_{q+1}\in \ell^*.
\]
For each $i$, choose a line $L_i\neq \ell^*$ through $P_i$.
Then the incidence submatrix on rows $P_1,\dots,P_{q+1}$ and columns $L_1,\dots,L_{q+1}$ has diagonal entries equal to $1$ and all off-diagonal entries equal to $0$.
In particular, any matrix supported exactly on incidences and nonzero on every incidence contains a nonsingular diagonal $(q+1)\times(q+1)$ submatrix, hence has rank at least $q+1$.
\end{lemma}

\begin{proof}
By construction $P_i\in L_i$, so the diagonal entries are incidences.
If $j\neq i$ and $P_j\in L_i$, then the two distinct points $P_i,P_j\in \ell^*$ would lie on both $L_i$ and $\ell^*$, contradicting the projective-plane axiom that two points determine a unique line.
Hence $P_j\notin L_i$ for every $j\neq i$, so all off-diagonal entries vanish.
The final rank claim follows because the resulting $(q+1)\times(q+1)$ submatrix is diagonal with nonzero diagonal entries.
\end{proof}

\begin{theorem}[No monomial rank-$3$ lift for $q\ge 6$]
\label{thm:no-monomial-rank3-lift}
Let $\Pi$ be a projective plane of order $q\ge 6$.
Then $I(\Pi)$ admits no monomial rank-$\le 3$ lift.
Equivalently, if $\Kapr(I(\Pi))\le 3$, then every rank-$\le 3$ lift of $I(\Pi)$ must contain some valuation-$0$ entry with a nontrivial first-order correction.
\end{theorem}

\begin{proof}
Assume for contradiction that there exists a monomial rank-$\le 3$ lift
\[
F(t)=U+tA
\]
as in Definition~\ref{def:monomial-lift-incidence}.
Since $F(t)$ has rank at most $3$, \cref{cor:first-order-coefficient-rank-bound} gives
\[
\rank_k(A)\le 6.
\]
On the other hand, $A$ is supported exactly on incidences and every incidence entry of $A$ is nonzero.
Applying \cref{lem:one-line-incidence-diagonal-block} to any line $\ell^*$ of $\Pi$, we obtain a diagonal $(q+1)\times(q+1)$ submatrix of $A$ with nonzero diagonal entries.
Therefore
\[
\rank_k(A)\ge q+1.
\]
For $q\ge 6$ this gives $\rank_k(A)\ge 7$, contradicting $\rank_k(A)\le 6$.
So no such monomial rank-$\le 3$ lift exists.
\end{proof}

\begin{remark}[Implication for first-order analysis]
\label{rem:note39-lambda-necessity}
\Cref{thm:no-monomial-rank3-lift} shows that for $q\ge 6$ the first-order obstruction cannot be resolved by incidence coefficients alone.
Thus any hypothetical rank-$\le 3$ lift must involve non-monomial first-order corrections on some valuation-$0$ entries.
In particular, any successful continuation of the analysis must control nontrivial first-order data on the off-incidence part of the lift, not only the incidence-supported coefficients.
\end{remark}

\begin{theorem}[Quadratic lower bound for first-order off-incidence support on a canonical line-choice block]
\label{thm:lambda-support-density-rank3}
Let $\Pi$ be a projective plane of order $q$, and assume $I(\Pi)$ admits a rank-$\le 3$ lift
\[
F(t)=U+tV+O(t^2).
\]
Fix one line $\ell^*$ with points $P_1,\dots,P_{q+1}$, and for each $i$ choose a line $L_i\neq \ell^*$ through $P_i$.
Set $N:=q+1$, and let
\[
V^{\triangle}:=(V_{P_i,L_j})_{1\le i,j\le N}.
\]
Then:
\begin{enumerate}[label=(\arabic*),leftmargin=2em]
  \item the diagonal entries of $V^{\triangle}$ are nonzero, because $P_i\in L_i$ implies $V_{P_i,L_i}\neq 0$ by \cref{rem:incidence-valuation-one-first-order};
  \item $\rank_k(V^{\triangle})\le 6$ by \cref{cor:first-order-coefficient-rank-bound};
  \item if one defines a graph $G$ on $\{1,\dots,N\}$ by declaring $\{i,j\}$ an edge whenever at least one of $V_{P_i,L_j}$ or $V_{P_j,L_i}$ is nonzero, then
  \[
  |E(G)|\ge \frac{N(N-6)}{12}=\frac{(q+1)(q-5)}{12}.
  \]
\end{enumerate}
In particular, every rank-$\le 3$ lift must activate at least
\[
\frac{(q+1)(q-5)}{12}
\]
unordered off-diagonal pairs inside this canonical $(q+1)\times(q+1)$ block, and hence at least that many nonincidence first-order corrections are forced somewhere among the entries $V_{P_i,L_j}$ with $i\neq j$.
\end{theorem}

\begin{proof}
The first item follows from \cref{rem:incidence-valuation-one-first-order}, because each diagonal position $(P_i,L_i)$ is an incidence.
The second item follows since $V^{\triangle}$ is a submatrix of $V$.

For the third item, let $S\subseteq\{1,\dots,N\}$ be an independent set in $G$.
Then for every distinct $i,j\in S$ we have both $V_{P_i,L_j}=0$ and $V_{P_j,L_i}=0$.
So the principal submatrix of $V^{\triangle}$ indexed by $S$ is diagonal with nonzero diagonal entries, hence has rank $|S|$.
Because $\rank(V^{\triangle})\le 6$, this forces $|S|\le 6$.
Equivalently, the complement graph $\overline G$ is $K_7$-free.
By Tur\'an's theorem,
\[
|E(\overline G)|\le \Bigl(1-\frac16\Bigr)\frac{N^2}{2}=\frac{5N^2}{12}.
\]
Therefore
\[
|E(G)|=\binom{N}{2}-|E(\overline G)|
\ge \frac{N(N-1)}{2}-\frac{5N^2}{12}
=\frac{N(N-6)}{12}.
\]
Finally, each edge of $G$ certifies that at least one of the two off-diagonal nonincidence entries $(P_i,L_j)$ or $(P_j,L_i)$ carries a nonzero first-order correction, yielding the stated lower bound on forced $\Lam$-activity.
\end{proof}

\begin{remark}[A support lower bound, not yet a contradiction]
\label{rem:note40-support-vs-solvability}
\Cref{thm:lambda-support-density-rank3} is a necessary support condition, not yet a contradiction.
It proves that the non-monomial regime forced by \cref{thm:no-monomial-rank3-lift} is quantitatively large: one canonical line-choice block already needs $\Omega(q^2)$ activated nonincidence corrections.
The remaining task is to combine this quadratic $\Lam$-budget with the many local equations $\Psi_4(U,A,W)=0$ so that the same variables are simultaneously numerous, structured, and rank-constrained.
\end{remark}




\section{Discussion and outlook}
\label{sec:discussion}



\subsection{Toward the Guterman--Shitov asymptotic problem}
The results proved here sharpen the shape of the rank-$3$ obstruction problem without claiming a full asymptotic resolution. By \cref{lem:incidence-as-cocircuit,cor:kapr3-representability-gate}, non-representability already gives the uniform lower bound $\Kapr(I(\Pi))\ge 4$, but that representability gate is only the starting point of the Guterman--Shitov question. The main unconditional advances of the present paper are the following: vanished identity-pattern minors force local cross-ratio defects \cref{cor:identity4-vanishing-forces-curvature,thm:flat-chart-obstruction}; counting those minors yields the global lower bound of \cref{thm:identity-global-curvature-pressure}; and first-order analysis yields both the no-monomial-lift theorem \cref{thm:no-monomial-rank3-lift} and the quadratic support bound of \cref{thm:lambda-support-density-rank3}. In particular, any hypothetical rank-$\le 3$ lift must already carry a large family of residue defects together with nontrivial first-order activity.

The remaining obstacle is therefore global rather than local. One still needs an independence principle strong enough to prevent a single rank-$3$ residue realization and a single admissible first-order direction from absorbing all overlapping local constraints simultaneously. A natural next target is to globalize the first-order equations attached to identity-pattern minors. In the monomial specialization $w_{ij}=0$ on off-incidence positions, \cref{cor:identity4-two-equation-package} reduces the first-order condition to
\[
\sum_{i=1}^4 a_i D_i(U)=0.
\]
A contradiction for this reduced system on a sufficiently independent family of identity-pattern minors would already eliminate the simplest surviving rank-$3$ scenarios.

\begin{conjecture}[First-order nonvanishing obstruction]
\label{conj:first-order-nonvanishing-obstruction}
For sufficiently large $q$, no rank-$3$ residue zero-pattern model $U$ satisfying the identity-pattern minor constraints developed in Sections~\ref{sec:framework}--\ref{sec:gadgets} admits a tangent direction $V\in T_U$ with
\[
V_{p\ell}\neq 0
\qquad\text{for every incidence }p\in\ell.
\]
Equivalently, any putative rank-$\le 3$ lift of $I(\Pi)$ should force at least one incidence entry to acquire valuation at least $2$.
\end{conjecture}

This conjecture is stated here, rather than in Section~\ref{sec:main}, to emphasize that the core theorem package of the paper is complete at the local and counting level; what is missing is a globalization theorem, not an additional local computation.



\appendix

\section{Technical lemmas and normalization tools}
\label{sec:app:tech}

\subsection{Rank bounds for first-order layers (details)}
\begin{lemma}[First-order rank bound (detailed version)]
If $F\in K^{m\times n}$ has $\rank(F)\le r$ and $F=U+tV+O(t^2)$, then $\rank(V)\le 2r$.
\end{lemma}

\begin{proof}
Choose a rank factorization
\[
F=A(t)B(t),
\qquad
A(t)\in K^{m\times r},\quad B(t)\in K^{r\times n}.
\]
Expand both factors to first order in the uniformizer $t$:
\[
A(t)=A_0+tA_1+O(t^2),
\qquad
B(t)=B_0+tB_1+O(t^2).
\]
Multiplying gives
\[
F=A_0B_0+t(A_1B_0+A_0B_1)+O(t^2).
\]
Comparing the coefficient of $t$ with the given expansion $F=U+tV+O(t^2)$, we obtain
\[
V=A_1B_0+A_0B_1.
\]
Now $A_1B_0$ has rank at most $r$, and $A_0B_1$ has rank at most $r$. Therefore
\[
\rank(V)\le \rank(A_1B_0)+\rank(A_0B_1)\le r+r=2r,
\]
as claimed.
\end{proof}

\subsection{Counting identity-pattern minors in the incidence matrix (details)}

\begin{lemma}[Detailed private-line count for identity-pattern minors]
\label{lem:identity-count-detailed}
Let $\Pi$ be a projective plane of order $q\ge 3$ with combinatorial incidence matrix $I(\Pi)$, and let $v:=q^2+q+1$.
The number of ordered $4\times4$ submatrices of $I(\Pi)$ equal to $I_4$ and obtained from ordered quadrangles with one private line through each vertex is at least
\[
v(v-1)q^2(q-1)^2(q-2)^4.
\tag{A.1}
\]
\end{lemma}

\begin{proof}
Choose an ordered quadrangle $(P_1,P_2,P_3,P_4)$ first.

There are $v$ choices for $P_1$ and $v-1$ choices for $P_2\neq P_1$.
The third point $P_3$ must avoid the line $\overline{P_1P_2}$.
Because every line has $q+1$ points, the number of available choices is
\[
v-(q+1)=q^2.
\]
Thus $(P_1,P_2,P_3)$ is already noncollinear.

For $P_4$, exclude the three lines
\[
\overline{P_1P_2},\qquad \overline{P_1P_3},\qquad \overline{P_2P_3}.
\]
Each contains $q+1$ points.
Since $P_1,P_2,P_3$ are noncollinear, those lines are pairwise distinct and meet pairwise exactly at $P_1,P_2,P_3$.
Hence their union has size
\[
3(q+1)-3=3q.
\]
So the number of admissible choices for $P_4$ is at least
\[
v-3q=(q^2+q+1)-3q=(q-1)^2.
\]
Any such $P_4$ yields an ordered quadrangle.

Now fix such a quadrangle.
For each $i\in\{1,2,3,4\}$, exactly $q+1$ lines pass through $P_i$.
Among them, precisely three also pass through one of the other vertices, namely
\[
\overline{P_iP_j}\qquad (j\neq i).
\]
These three lines are distinct because no three vertices are collinear.
Therefore exactly
\[
(q+1)-3=q-2
\]
lines through $P_i$ avoid the other three vertices.
Choosing one such line $L_i$ for each $i$ produces a $4\times4$ incidence block with
\[
P_i\in L_i,\qquad P_j\notin L_i\quad (j\neq i),
\]
hence the block is exactly $I_4$.

Multiplying the four point-selection factors and the four private-line factors gives
\[
v(v-1)\cdot q^2\cdot (q-1)^2\cdot (q-2)^4,
\]
which is \textup{(A.1)}.
\end{proof}

\subsection{Multiplicity of one zero rectangle inside identity-pattern minors (details)}

\begin{lemma}[Detailed multiplicity bound for a fixed $0$-rectangle]
\label{lem:identity-rectangle-multiplicity-detailed}
Let $\Pi$ be a projective plane of order $q\ge 3$, and let
\[
R=(\{p,q\}\times \{\ell,m\})
\]
be a $0$-rectangle in $I(\Pi)$, so $p,q\notin \ell\cup m$ and $\ell\neq m$.
Then the number of ordered identity-pattern minors from the private-line construction of \cref{lem:private-line-identity-minors} that contain $R$ as one of their admissible complement rectangles is at most
\[
C(q+1)^4
\tag{A.2}
\]
for an absolute constant $C$.
One may take $C=16(4!)^2$.
\end{lemma}

\begin{proof}
Fix one such minor $M$ containing $R$.
Inside any ordered $4\times4$ identity-pattern block, the admissible $0$-rectangles are exactly the six complement rectangles obtained by choosing two point rows and the two line columns matched to the complementary two points.
Therefore, if $R$ occurs inside $M$, then there exist two further points $r,s$ in the block such that $\ell$ and $m$ are precisely the private lines matched to $r$ and $s$.
In particular,
\[
r\in \ell\setminus\{\ell\cap m\},
\qquad
s\in m\setminus\{\ell\cap m\}.
\]
Each of the lines $\ell,m$ contains exactly $q+1$ points, so there are at most $(q+1)^2$ choices for the ordered pair $(r,s)$.
(Indeed, one can sharpen this to $q^2$, but the coarser bound is sufficient.)

Once $(p,q,r,s)$ are fixed, any compatible ordered identity minor must assign private lines to $p$ and $q$ as well.
For each of those two points there are at most $q+1$ possible lines through the point, hence at most $(q+1)^2$ choices for the ordered pair of private lines attached to $p$ and $q$.
Some of these choices fail to produce an actual identity-pattern minor because they may meet one of the other three points, but counting them only enlarges the bound.

Finally, after the underlying four points and four lines are chosen, only constantly many orderings of the rows and columns realize the prescribed rectangle $R$ as one of the six complement rectangles of the block.
A crude bound is obtained by freely ordering the four point rows and four line columns and then multiplying by an extra factor $16$ for the possible identifications of the two rows and two columns appearing in $R$; this gives the explicit constant $16(4!)^2$.
Absorbing all such ordering choices into one constant $C$ yields the claimed upper bound \textup{(A.2)}.
\end{proof}


\section{Supplementary computational checks}
\label{sec:app:computations}


\subsection{Recurring local configurations (supplementary)}

This subsection records the three local configurations that recur most persistently in the later arguments.

\begin{example}[$4\times4$ identity pattern]
For the $4\times4$ identity valuation pattern, the valuation-minimal permutations are exactly the nine derangements of $\{1,2,3,4\}$. Thus a vanished identity-pattern minor forces a multi-term residue cancellation rather than a binomial relation. This is the local source of the derangement equation from Section~\ref{sec:gadgets}.
\end{example}

\begin{example}[Degenerate square-minimizer pattern]
A degenerate diamond may force four valuation-minimal determinant terms rather than three. In that regime the initial form factors as a sum of two binomial remainders, which is why the degenerate analysis naturally leads to transport laws and cancellation-depth considerations rather than to a single trimonial identity.
\end{example}

\begin{example}[Bridge-cycle holonomy]
For the skew rectangle $(A,B;L_0,L_1)$ extracted from a strict $B_*$ configuration, the bridge-cycle holonomy is exactly the residue cross-ratio
\[
\rho_{AB}^{L_0L_1}=\frac{u_{A,L_1}}{u_{A,L_0}}\frac{u_{B,L_0}}{u_{B,L_1}}.
\]
Equivalently, the associated skew determinant can be written as
\[
\Delta_{AB}^{L_0L_1}=u_{A,L_0}u_{B,L_1}\bigl(1-\rho_{AB}^{L_0L_1}\bigr),
\]
so the same local datum can be read either as a cross-ratio defect or as a holonomy defect.
\end{example}

\subsection{\texorpdfstring{Small-$q$ incidence data and computational checks}{Small-q incidence data and computational checks}}
The statements in this subsection are finite computational observations for $\mathrm{PG}(2,2)$ and $\mathrm{PG}(2,3)$. They are included as evidence about the first-layer cycle constraints, not as substitutes for the general argument.
The finite claims below come from exhaustive enumeration of $4\times4$ submatrices of the combinatorial incidence matrices of $\mathrm{PG}(2,2)$ and $\mathrm{PG}(2,3)$, followed by direct computation of the minimizing permutations and the associated cycle vectors.

\begin{proposition}[Computational cycle-space generation by $(0,2)$-type $4\times4$ minors for $q=2,3$]
\label{prop:q23-cycle-generation}
Let $\Pi_q$ be the Desarguesian projective plane of order $q\in\{2,3\}$, with combinatorial incidence matrix $I_q$ and nonincidence
graph $G_0(I_q)$. Let $\mathcal S_{0,2}$ be the set of all $4\times4$ submatrices $B$ of $I_q$ such that:
\begin{enumerate}[leftmargin=2em]
  \item $\min_{\sigma\in S_4} \sum_{i=1}^4 B_{i,\sigma(i)} = 0$;
  \item exactly two permutations attain that minimum.
\end{enumerate}
For each $B\in \mathcal S_{0,2}$, let $c(B)\in Z_1(G_0(I_q);\mathbb F_2)$ be the cycle-vector of the symmetric difference of its two
weight-$0$ perfect matchings. Then
\[
\operatorname{span}_{\mathbb F_2}\{c(B): B\in\mathcal S_{0,2}\}
=
Z_1(G_0(I_q);\mathbb F_2).
\]
Numerically,
\[
\dim Z_1(G_0(I_2);\mathbb F_2)=15,
\qquad
\dim Z_1(G_0(I_3);\mathbb F_2)=92,
\]
and exhaustive enumeration of all $(0,2)$-type $4\times4$ minors reaches exactly those dimensions.
\end{proposition}

\begin{proof}
We generated $I_q$ explicitly for $q=2,3$ using the standard Desarguesian model $\mathrm{PG}(2,q)$, with points and lines represented by
$1$-dimensional subspaces of $\mathbb F_q^3$ and incidence given by the projective dot-product relation. For each choice of
$4$ points and $4$ lines, we evaluated all $24$ permutations, retained exactly those submatrices with tropical minimum $0$ attained by
exactly two permutations, formed the symmetric-difference cycle-vector of the two corresponding zero-matchings, and computed the
$\mathbb F_2$-span of those vectors by Gaussian elimination.

For $q=2$, the nonincidence graph has $|V|=14$ and $|E|=28$, so
\[
\dim Z_1(G_0(I_2);\mathbb F_2)=|E|-|V|+1=15.
\]
The exhaustive scan finds $777$ minors of type $(0,2)$, and their cycle-vectors span dimension $15$.

For $q=3$, the nonincidence graph has $|V|=26$ and $|E|=117$, so
\[
\dim Z_1(G_0(I_3);\mathbb F_2)=|E|-|V|+1=92.
\]
The exhaustive scan finds $63{,}414$ minors of type $(0,2)$, and their cycle-vectors span dimension $92$.
Therefore the displayed equality holds in both cases.
\end{proof}

\begin{proposition}[Computational cycle-type counts for $(0,2)$-type $4\times4$ minors when $q=2,3$]
\label{prop:q23-cycle-type-counts}
For $\Pi_q=\mathrm{PG}(2,q)$ with $q\in\{2,3\}$, every $(0,2)$-type $4\times4$ minor has a unique cycle of length $4$, $6$, or $8$.
The exhaustive counts are:
\[
q=2:\qquad \#C_4=336,\quad \#C_6=420,\quad \#C_8=21,
\]
and
\[
q=3:\qquad \#C_4=37{,}908,\quad \#C_6=24{,}804,\quad \#C_8=702.
\]
In particular, the $6$-cycle case occurs abundantly in both base planes.
\end{proposition}

\begin{proof}
The same exhaustive scan used in \cref{prop:q23-cycle-generation} records, for each $(0,2)$-type minor, the cycle length of the
symmetric difference of its two zero-matchings. By \cref{lem:02-single-cycle}, only the lengths $4$, $6$, and $8$ can occur, and the
listed counts are exactly what the computation returns.
\end{proof}

\begin{corollary}[Characteristic-$2$ specialization of the first layer for $q=2,3$]
\label{cor:q23-gauge}
Let $q\in\{2,3\}$, and let $u_{p,\ell}\in k^*$ be the leading coefficients on the nonincidence edges of a lift of $I_q$.
Assume that $\operatorname{char}(k)=2$ and that every $(0,2)$-type $4\times4$ minor satisfies its induced determinant cancellation
constraint. Then there exist functions
\[
\alpha:\mathcal P\to k^*, \qquad \beta:\mathcal L\to k^*
\]
such that
\[
u_{p,\ell}=\alpha_p\beta_\ell
\]
for every nonincidence edge $(p,\ell)$.
\end{corollary}

\begin{proof}
By \cref{prop:02-signed-holonomy}, each $(0,2)$-type minor imposes $\operatorname{Hol}_{u}(C)=(-1)^k$ on its associated cycle.
In characteristic $2$, one has $(-1)^k=1$ for every $k$, so every such constraint becomes ordinary holonomy-$1$.
Now combine \cref{prop:q23-cycle-generation} with \cref{lem:g0-connected,lem:holonomy-factorization}.
\end{proof}

\begin{remark}[Interpretation of the finite enumerations]
For $q=2,3$, the $(0,2)$-type minors generate the full $\mathbb F_2$-cycle space, but the induced determinant constraints are sign-sensitive: $4$- and $8$-cycles impose holonomy $+1$, whereas $6$-cycles impose holonomy $-1$. The finite calculations therefore support the view that the correct first-layer object is a signed-holonomy system rather than an unsigned cycle condition. What remains open is the general-$q$ counterpart of this signed generation phenomenon and its interaction with the genuinely multi-term cancellations treated earlier in the paper.
\end{remark}


\bibliographystyle{amsplain}
\bibliography{refs}

\end{document}